%% file: main.tex
\documentclass[american,letterpaper]{article}
\usepackage{import}
\subimport{}{preamble.tex}

\usepackage[
style=authoryear-icomp,
natbib=true,
url=true, 
doi=false,
ibidtracker=false,
eprint=false,
dashed=false,
maxbibnames=99
]{biblatex}
\usepackage[affil-it]{authblk}
\title{Bootstrap Robust Prescriptive Analytics}
\author[1]{Dimitris Bertsimas\thanks{\href{mailto:dbertsim@mit.edu}{dbertsim@mit.edu}}}
\author[1]{Bart \mbox{Van Parys}\thanks{\href{mailto:vanparys@mit.edu}{vanparys@mit.edu}}}
\affil[1]{Operations Research Center, Massachusetts Institute of Technology}

\providecommand{\keywords}[1]{\textbf{\textit{Keywords: }} #1}

\begin{document}
\setlength{\baselineskip}{1.5em}

\maketitle

\begin{abstract}
  We address the problem of prescribing an optimal decision in a framework where the cost function depends on uncertain problem parameters that need to be learned from data.
  Earlier work proposed prescriptive formulations based on supervised machine learning methods.  These prescriptive methods can factor in contextual information on a potentially large number of covariates to take context specific actions which are superior to any static decision. When working with noisy or corrupt data, however, such nominal prescriptive methods can be prone to adverse overfitting phenomena and fail to generalize on out-of-sample data. 
  In this paper we combine ideas from robust optimization and the statistical bootstrap to propose novel prescriptive methods which safeguard against overfitting.
  We show indeed that a particular entropic robust counterpart to such nominal formulations guarantees good 
  performance on synthetic bootstrap data. As bootstrap data is often a sensible proxy to actual out-of-sample data, our robust counterpart can be interpreted to directly encourage good out-of-sample performance.
  The associated robust prescriptive methods furthermore reduce to convenient tractable convex optimization problems in the context of local learning methods such as nearest neighbors and Nadaraya-Watson learning.
  We illustrate our data-driven decision-making framework and our novel robustness notion on a small newsvendor problem.
\end{abstract}


\section{Introduction}

Most practical decisions need to be made despite the fact that their cost may be uncertain. Indeed, often the cost $L(z, y)$ of any potential decision $z$ depends on an unknown stochastic parameter $y$. Stocking goods under uncertain customer demand or devising profitable investment strategies when facing volatile returns are two practical examples. Stochastic optimization aims to minimize expected cost, i.e., 
\begin{equation}
  \label{prob:classical}
  z^\star \in \arg\min_{z} \,\E{Y^\star}{L(z, y)}
\end{equation}
where the random variable $y$ is distributed as $Y^\star$. We assume here implicitly that the minimization is carried out only over those decisions $z$ for which the cost function is well defined, i.e., $\dom( \E{Y^\star}{L(z, y)})$ . To avoid technicalities we make the standard assumption that the decision $z \in \Re^{\dim(z)}$ and the random variable $y \in \Re^{\dim(y)}$ take values in finite dimensional vector spaces.

Such stochastic optimization formulation makes sense when nothing beyond the distribution of the stochastic variable $y$ is known before a decision is to be made. Often however, additional information concerning a large number of covariates $x$ can be obtained \textit{before} we need to make our decision. After observing a particular context $x=\obs$, the decision should take this additional information into account and the problem in need of attention should not be the classical stochastic optimization problem \eqref{prob:classical} but rather
\begin{equation}
  \label{prob:data-analytics}
  z^\star(\obs) \in \arg\min_{z} \, \E{D^\star}{L(z, y) | x=\obs}
\end{equation}
where we assume that $\obs\in \Re^{\dim(x)}$ consists of all relevant contextual information. We assume that the context $x$ is a random variable with unknown distribution $X^\star$ whereas $\obs$ is one particular observed context. The distribution $D^\star$ represents here the joint distribution of both the parameter $y$ as well as the covariates $x$ whereas $\E{D^\star}{L(z, y) | x=\obs}$ denotes the expected cost of decision $z$ conditioned on the event $x=x_0$. Denote with $\mc D$ the set of all distributions supported on the set $\Omega\subseteq \Re^{\dim(x)\times\dim(y)}$ of all possible parameter and covariate realizations, then clearly $D^\star\in \mc D$. The optimal contextual decision $z^\star(\obs)$ has minimal cost as measured by the expected value of its loss $L(z^\star(\obs), y)$ conditioned\footnote{Technically, the conditional expectation $\E{D^\star}{L(z, y) | x}$ is a random variable and uniquely defined only up to events of measure zero. That is, the inclusion in \eqref{prob:data-analytics} can hold merely almost surely; see for instance the standard text by \citet{billingsley2008probability}. Slightly abusing notation all statements in this paper involving the observation $x_0$ should be interpreted to hold hence as $X^\star$-almost surely where $X^\star$ is the distribution of the covariates $x$.}
on all observed contextual information.
Unfortunately, distributions are never observed directly but rather must be learned from historical observations.

Only historical data concerning both uncertain parameters and contexts is typically available in practice. Data instead of distributions should hence be the primitive when making decisions under uncertainty. To that end we consider the training data set
\begin{equation}
  \label{data:supervised}
  \train \defn  [\bar d_1 = (\bar x_1, \bar y_1), \bar d_2 = (\bar x_2, \bar y_2), \dots].
\end{equation}
We use the symbol ``$\bar{\,\cdot\,}$'' throughout this paper to emphasize that the training data is here deterministic rather than random.  Data consisting of observations of the parameter $y$ as well as the context $x$ is often denoted as supervised training data.
Typically only a limited number of historical observations can be used as training data. We will denote with $\train[n]=[\bar d_1 = (\bar x_1, \bar y_1), \dots, \bar d_n = (\bar x_n, \bar y_n)]$ the training data set consisting of the first $n$ observations in Equation \eqref{data:supervised}.
We will assume that the data comes without any time dependence, or if it does, time is explicitly included as an additional covariate.
Consequently, most appropriate learning methods will be insensitive to the order of the data points. Hence, the information contained in the data set can be represented compactly through its empirical distribution $$D_{\train[n]} \defn \textstyle\frac 1n \sum_{(\bar x, \bar y)\in \train[n]}\delta_{(\bar x, \bar y)}.$$ To allow for a statistical analysis of data-driven formulations it is convenient to assume that the training data is obtained as independent samples from the unknown distribution $D^\star$, i.e., it is generated by
\begin{equation}
  \label{data:resampled}
  \data \defn [d_{1} = (x_{1}, y_{1})\iid D^\star,~d_{2} = (x_{2}, y_{2})\iid D^\star,~ \dots] \sim D^{\star\infty}.
\end{equation}
Here, we make a distinction between $D^{\star\infty}$ as the data generating process and $D^\star$ the unknown joint distribution between the parameters $y$ and covariates $x$. The training set is hence supported on the empirical support set $\Omega_n$ which is a subset of $\Omega$. When we denote with $\mc D_n$ the set of all distributions supported on the set $\Omega_n$, then clearly $D_{\train[n]}\in \mc D_n\subseteq \mc D$.

\subsection{Data-Driven Formulations}
\label{ssec:data-driv-form}

In this section we briefly position our approach among several distinct alternatives to decision-making with data. Making decisions based on data has received considerable attention in the literature. 
We restrict attention to three distinct data-driven approaches which result in computationally attractive formulations.

\textit{Empirical-Risk-Minimization:}
The prescription problem \eqref{prob:data-analytics} can alternatively be defined through the risk minimization formulation $z^\star(\cdot)\in \arg \min_{z\in \mc F} \, \E{D^\star}{L(z(x), y)}$ over the set of all functions $\mc F$ mapping covariates to decisions. One could hence consider
\begin{equation}
  \label{eq:emp_risk}
  z^{}_{\train[n]}(\cdot) \in \arg\min_{z(\cdot)\in \mc C} ~\frac{1}{n}\textstyle\sum_{(\bar x, \bar y)\in \train[n]}L(z(\bar x), \bar y)
\end{equation}
where $\mc C$ is a set of functions with suitable properties. Here $\mc C$ must be a strict subset of the set of all functions $\mc F$ as to avoid overfitting and encourage generalization. \citet{rudin2014big} consider the class $\mc C= \mc F_{\lin}$ consisting of all linear functions. The empirical risk minimization problem \eqref{eq:emp_risk} then reduces to a convex problem in the coefficients describing the linear functions in $\mc F_{\lin}$. \citet{rudin2014big} also establish out-of-sample guarantees in the particular context of a newsvendor problem. \citet{bertsimas2014predictive} prove similar out-of-sample guarantees for the more general prescription problems considered here and show in particular that for any $\delta>0$ the probability of drawing random data according to Equation (\ref{data:resampled}) and observing
\begin{equation}
  \label{eq:nk-oos}
  \E{D^\star}{L(z(x), y)}\leq \frac{1}{n}\textstyle\sum_{(\bar x, \bar y)\in \data[n]}L(z(\bar x), \bar y)+\bar L\sqrt{\log(1/\delta)/2n}+\bar G R_n(\mc C)
\end{equation}
is at least $1-\delta$ for all functions $z(\cdot)\in \mc C$. The previous inequality bounds, if $z^\star(\cdot)\in \mc C$, the difference between the empirical average cost and the actual expected cost of not only the optimal prescription but in fact all feasible functions in $\mc C$. The constants $\bar L$ and $\bar G$ are taken to satisfy $\sup_{z, \, y}L(z, y)\leq \bar L$ and $\sup_{z\neq z',\,y} \tfrac{(L(z, y)-L(z', y))}{\norm{z-z'}_\infty}\leq \bar G$, respectively. The Rademacher complexity is a measure of the size of the class of considered prescriptive functions $\mc C$ and is defined as $$R_n(\mc C) \defn \E{}{\textstyle\frac{2}{n} \sup_{z(\cdot)\in \mc C} \sum_{ d\in \data[n]}\sum_{k=1}^{\dim(x)} \sigma_{ d, k} z_k(\bar x)}$$
where $\sigma_{\bar d, k}$ are independent and equiprobably $\{+1, -1\}$ and the data is generated as in Equation (\ref{data:resampled}). \citet{bertsimas2014predictive} show furthermore that $R_n(\mc F_\lin) \leq \mc O(\tfrac{1}{\sqrt{n}})$ and hence, as expected, when more data is observed stronger out-of-sample guarantees can be established.

By explicitly constraining the considered prescriptors to be elements of $\mc C$, however, bias is introduced when $z^\star(\cdot)\not\in \mc C$. There is indeed no reason in general to expect that $z^\star(\cdot)$ should possess for instance a linear parametric structure. As pointed out by \citet{bertsimas2014predictive} constraints on the decision may give rise to nonlinearities which are particularly challenging to deal with. \citet{rudin2014big} do consider nonlinear prescriptions $z_{\train[n]}(\cdot)$ implicitly via the introduction of auxiliary nonlinear transformed covariates. In this paper we will however focus on formulations of an entirely different nature instead.

\textit{Estimate-First-Optimize-Later:}
An alternative to the one-shot approach taken by empirical risk minimization formulations is to separate the problem of prescription with covariate information into an estimation and a subsequent optimization step as visualized in Figure \ref{fig:estimate-first-optimize-later}.
In the estimation step an estimator\footnote{Our notation here alludes to the fact that this estimator function mapping data to cost estimates is often asymptotically with $n\to\infty$ an unbiased estimate of the actual unknown expected cost $z\mapsto \E{D^\star}{L( z, y)|x=\obs}$. In our paper the symbol ``$\rm{E}$'' will be associated with an estimator (a procedure mapping data to cost estimates) while ``$\mb E$'' denotes an expectation operator which may define for instance the actual but unknown expected cost.}
$z\mapsto \Es{D_{\train[n]}}{L( z, y)|x=\obs}$ is considered which based on the provided training data estimates the cost of any decision in the context of interest. Such estimators can be obtained using a variety of machine learning methods. \citet{hannah2010nonparametric} consider the classical kernel smoothing method proposed independently by \citet{watson1964smooth, nadaraya1964estimating}. \citet{bertsimas2014predictive} consider additional formulations based on nearest neighbors learning discussed in \citet{altman1992introduction}, and regression trees and random forest learning proposed by \citet{breiman2001random}. We also want to mention that \citet{elmachtoub2021smart} refer to this two-step approach as ``predict, then optimize'' and discuss several of its limitations and suggest improvements.

Estimate-first-optimize-later formulations subsequently prescribe the action
\begin{equation}
  \label{eq:eto}
  z_{\train[n]}(\obs) \in \arg \min_{z}\,\Es{D_{\train[n]}}{L(z, y)|x=\obs}.
\end{equation}
We will consider in this paper estimate-first-optimize-later formulations based on local learning methods such as Nadaraya-Watson and nearest neighbors learning. In Section \ref{sec:prescriptive-analytics} we in fact indicate that both can be generalized using the variable metric estimator discussed by \citet{sain2002multivariate}. \citet{bertsimas2014predictive} prove that both the Nadaraya-Watson and nearest neighbors formulations are asymptotically consistent under mild assumptions on the training data and loss function. This in stark contrast to empirical risk minimization formulations which are biased whenever $z^\star(\cdot) \not \in \mc C$. Empirical risk minimization formulations on the other hand provide through the complexity of the set $\mc C$ an explicit way to remedy adverse overfitting effects. Estimate-first-optimize-later formulations, although often empirically observed to perform well out-of-sample by \citet{rudin2014big, ,bertsimas2014predictive}, come without such an immediate remedy when they do overfit. It is indeed well known that unbiased estimators based on Nadaraya-Watson and nearest neighbors learning typically suffer a large variance in particular when the covariate dimension is large; see \citet{friedman2001elements}[Section 2.5]. Subsequent minimization of the unbiased estimator $z\mapsto\Es{D_{\train[n]}}{L(z, y)|x=\obs}$ only amplifies this issue and can results in overly optimistic prescriptors as pointed out by \citet{michaud1989markowitz} and \citet{rudin2014big}. The performance of estimate-first-optimize-later formulations on out-of-sample data can in such settings be very poor.

\begin{figure}
  \centering
  \begin{tikzpicture}[node distance=7cm]

    \node[input, name=data] {};
    \node [block, right of=data, node distance=3.5cm] (predictor) {Estimator};
    \draw[->] (data) -- node[above] {Data} (predictor);
    \draw[->] (data) -- node[below] {$D_{\train[n]}$} (predictor);
    \node [block, right of=predictor] (opt) {Optimizer};
    \draw [->] (predictor) -- node[above] {Cost} (opt);
    \draw [->] (predictor) -- node[below] {$\Es{D_{\train[n]}\!}{L(z, y)|x=\obs}$} (opt);
    \node [right of=opt, node distance=4cm] (dec) {};
    \draw[->] (opt) -- node[above] {Decision} (dec);
    \draw[->] (opt) -- node[below] {$z_{\train[n]}(\obs)$} (dec);
    
  \end{tikzpicture}
  \caption{Estimate-first-optimize-later formulations construct first a cost estimator $z\mapsto\Es{D_{\train[n]}\!}{L(z, y)|x=\obs}$ and subsequently optimize for a decision promising minimum (estimated) cost as in Equation (\ref{eq:eto}). }
  \label{fig:estimate-first-optimize-later}
\end{figure}

\textit{Robust-Budget-Minimization: } It is clear that when given only a limited amount of training data, data-driven formulations must be guarded against overfitting to one specific training data set. Overfitting can be discouraged by minimizing a regularized objective function
\begin{equation}
  \label{eq:bm}
  z_{\train[n]}(\obs)\in\arg\min_z \, \{c_{n}(z, D_{\train[n]}, \obs)= \Es{D_{\train[n]}}{L(z, y)|x=\obs} + J_{\train[n]}(z, \obs)\}
\end{equation}
consisting of both an asymptotically unbiased cost estimate and an additional non-negative regularization term. The regularization term  makes the budgeted cost of any decision conservative but hopefully ensures decisions whose actual costs do not break the estimated budget. We will  denote in this paper events in which the estimated budget of the suggested action does not cover its actual cost as \textit{disappointing}. One way of regularizing is by considering a robust counterpart formulation
\begin{equation}
  \label{eq:dro}
  z_{\train[n]}(\obs)\in\arg\min_z \, \{c_{n}(z, D_{\train[n]}, \obs)= \textstyle\sup_{D \in \mc A(D_{\train[n]})} \,\Es{D}{L(z, y)|x=\obs}\}
\end{equation}
where the ambiguity set $\mc A(D_{\train[n]})$ typically consists of all distributions consistent with the observed training data instead of merely the empirical distribution $D_{\train[n]}$. As the robust counterpart is with respect to a set of distributions, formulation (\ref{eq:dro}) takes a so-called distributionally robust perspective.

Distributionally robust optimization has attracted significant attention in the context of classical stochastic optimization without contextual information as it can present a disciplined safeguard mechanism against overfitting. By using a robust counterpart with respect to an ambiguity set of distributions around an estimated nominal one, they were shown by \citet{vanparys2017data} to be minimally biased while still enjoying statistical out-of-sample guarantees. Many interesting choices of the ambiguity set furthermore result in a tractable overall decision-making approach. The ambiguity set can be defined, for example, through confidence intervals for the distribution's moments as done by \citet{popescu2005semidefinite, delage2010distributionally, zymler2013distributionally,wiesemann2014distributionally, vanparys2016generalized}. Alternatively, \citet{wang2009likelihood} use an ambiguity set that contains all distributions that achieve a prescribed level of likelihood, while \citet{bertsimas2014robust} based theirs on distributions which pass a statistical hypothesis test. Distance-based ambiguity sets which contain all distributions sufficiently close to a reference with respect to probability metrics \citep{ben2013robust} such as the Prokhorov metric \citep{erdogan2006}, the Wasserstein distance \citep{pflug2007,esfahani2015data}, or the total variation distance \citep{sun2016} have received a lot of attention in the recent literature. Surprisingly, the merits of distributional robustness in the context of prescription problems with contextual information has not been studied in depth. One major contribution of this paper is to show that the distributional robust perspective applies just as well to prescription problems with contextual information. Furthermore, we show in Section \ref{sec:bootstrap-performance} that for a judiciously chosen ambiguity set $\mc A(D_{\train[n]})$ the resulting robust formulations enjoy interpretable performance guarantees in the context of bootstrapped ``out-of-sample'' data.

\subsection{Motivating Example} \label{sec:newsvendor}

Before we formally introduce our novel robustness notion based on bootstrap aggregation, we motivate its main idea with the help of an intuitive example which we will explore further in Section \ref{ssec:news-vendor}. We single out the {newsvendor problem, which is a common model in economics and operations management, for its simplicity. Consider a newspaper vendor who must decide how many copies $z$ to stock each day in the face of uncertain customer demand $y$ well knowing that any unsold copies will be worthless at the end of each day.
Ideally, the newsvendor would like to order exactly $z=y$ newspaper as they are a perishable good. Unfortunately, a decision on the order quantity needs to be made before the demand is observed.
Fortunately, however, the newsvendor can observe before making the order several covariates $x=\obs$ which may correlate with the uncertain demand.
A sensible goal is to order a quantity that minimizes the total expected cost according to 
\begin{equation}
  \label{prob:news-vendor}
  \begin{array}{rl}
  z^\star(\obs)\in\arg\inf_z & \E{D^\star}{ b \cdot (y-z)^{+ } + h \cdot (z-y)^{+ }  \big| x = \obs}
  \end{array}
\end{equation}
where $D^\star$ is the joint distribution between costumer demand and observed covariates.
The constants \(b \in \Re_+\) and \(h \in \Re_+\) represent here the marginal opportunity cost of ordering too few newspapers and the loss in revenue associated with ordering surplus newspapers. If the distribution $D^\star$ is known, then a classical result states that the optimal decision is given by the conditional quantile
\begin{equation}
  \label{eq:sol-news-vendor}
   z^\star(\obs) := \inf \set{z\in\Re}{ \E{D^\star}{\mathbb 1\{y\leq z\} \big | x = \obs } \geq \tfrac{b}{(b+h)}}
 \end{equation}
of the demand distribution in the covariate context of interest.

The classical newsvendor formulation assumes the joint distribution \( D^\star \) between demand and covariates to be known. In practice however distributions are never observed directly but rather must be estimated from a finite amount of training data.
Let us for the sake of exposition single out one particular formulation based on nearest neighbors learning. That is, we consider
\begin{equation}
  \label{eq:data-driven-newsvendor}
  z_{\train[n]}(\obs) := \arg\min_z ~\left\{c_{n}(z, D_{\train[n]}, \obs) = \frac{1}{k}\textstyle\sum_{(\bar x, \bar y)\in N^k_n(\obs)} b \cdot (\bar y-z)^{+ } + h \cdot (z-\bar y)^{+ }\right\}
\end{equation}
where $N^k_n(\obs)$ denotes here the $k$ nearest neighbors to the covariate of interest in the provided training data set. Hence, our decision minimizes the budgeted cost as estimated by a classical nearest neighbors learning method. In Figure \ref{fig:motivational_example} we visualize this cost estimate as a function of the context of interest to illustrate the nature of the nearest neighbors estimator.

Remark that our data-driven decision $z_{\train[n]}(\obs)$ only minimizes the \textit{predicted} cost in the context of interest rather than the unknown out-of-sample \textit{actual} cost. This implies that once the unknown cost realizes it may exceed the predicted cost budgeted by the newsvendor. Clearly, such disappointment events are very undesirable and may result in financial difficulties for the newsvendor who must now find additional funds to cover the difference. A standard practice in machine learning to quantify the probability of disappointment events is to consider the performance of our decision on validation data. Suppose indeed we could obtain access to independent validation data sets $\validate[n]\in V$ each containing the same number of data points as the training data set. A practical measure for disappointment would be the fraction
\begin{equation}
  \label{eq:val-disappointment}
  \frac{1}{{\abs{V}}}{\textstyle\sum_{\validate[n]\in V} \mb 1\{\Es{D_{\validate[n]}}{L(z_{\train[n]}(\obs), y)|x=\obs}\geq c_{n}(z_{\train[n]}(\obs), D_{\train[n]}, \obs)\}}
\end{equation}
of all validation sets based on which an independent cost estimate for our decision $z_{\train[n]}(\obs)$ breaks our estimated original budget.
Notice that the previous metric critically relies on our ability to obtain large quantities of validation data which in practice may be impossible or very expensive to obtain. We want a measure of disappointment for our cost estimate without access to such independent validation data. 

We could try to make up for not having access to independent validation data by considering a resample method instead.
Bootstrapping according to \citet{efron1982jackknife} is the process of resampling with replacement a proxy validation data set from the original training data. This resampling procedure is described formally as the stochastic process
\begin{equation}
  \label{data:bootstrapped}
  \bootstrap \defn [(x_{1, r}, y_{1, r})\iid D_{\train[n]},  M_{2, r} = (x_{2, r}, y_{2, r})\iid D_{\train[n]},\dots]\sim D_{\train[n]}^\infty
\end{equation}
of resampling independent data points from the empirical distribution of the training data. Again we make a distinction between the bootstrap data generation process $D_{\train[n]}^\infty$ and the empirical distribution $D_{\train[n]}$ of the training data.
The bootstrap method hence generates its own synthetic validation data sets directly from the training data. Bootstrap data sets so obtained are merely proxies to actual validation data which ought to be drawn from the stochastic process stated in (\ref{data:resampled}) which generated the training data. In the context of budget minimization formulations, we can sample bootstrap data sets $\bootstrap[n] \in B$ and use the fraction
\begin{equation}
  \label{eq:bootstrap-approx}
  \frac{1}{\abs{B}}{\textstyle\sum_{\bs[n]\in B} \mb 1\{\Es{D_{\bootstrap[n]}}{L(z_{\train[n]}(\obs), y)|x=\obs}\geq c_{n}(z_{\train[n]}(\obs), D_{\train[n]}, \obs)\}}
\end{equation}
as a proxy to the out-of-sample disappointment \eqref{eq:val-disappointment} on actual validation data. As bootstrapping does not require access to any validation data it is a practically viable approach to quantify the out-of-sample performance of budget minimization formulations; see also Figure \ref{fig:motivational_example}.

\begin{figure}
  \begin{subfigure}{0.45\textwidth}
    \centering
    \includegraphics[width=\textwidth]{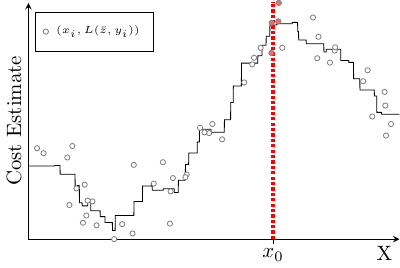}
    \caption{Nominal Cost}
  \end{subfigure}%
  \hfill
  \begin{subfigure}{0.45\textwidth}
    \centering
    \includegraphics[width=\textwidth]{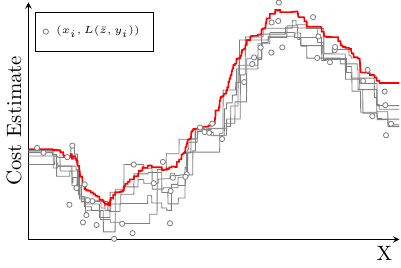}
    \caption{Bootstrap Robust Cost}
  \end{subfigure}
  \caption{Figure (a) illustrates the cost budgeted for a given decision $\bar z$ as a function of the context of interest $x=\obs$ using the nearest neighbors formulation \eqref{eq:data-driven-newsvendor}  with $k=4$. That is, the cost estimate at the context of interest is the average of the empirical costs $L(\bar z, y_i)$ associated with the four training data points $(x_i, y_i)\in N^k_n(\obs)$ indicated in red closest to the context of interest. As these nearest neighbors $N^k_n(\obs)$ are very sensitive to the particular context of interest the associated cost function indicated in black is discontinuous. Figure (b) depicts the associated cost estimates in gray using the nearest neighbors formulation based on five random bootstrap data sets. Evidently, as each bootstrap data set is slightly different the associated cost estimates will differ slightly as well. Definition \ref{def:bootstrap-robust} calls a cost estimate bootstrap robust if only a small fraction $b$ of all such bootstrapped cost estimates would break budget. We depict in red one such robust cost estimate ($b=0.05$) computed empirically as the corresponding quantile of the cost estimates based on $\abs{B}=2000$ bootstrap data sets. When these bootstrap data sets are good proxies for actual out-of-sample data and if nearest neighbors is an appropriate learning method then the newsvendor can reasonably hope that the associated robust prescription and budget will not disappoint.}
  \label{fig:motivational_example}
\end{figure}

\subsection{Bootstrap Aggregation}

In what follows we will in fact consider the limiting bootstrap out-of-sample disappointment in which the number of bootstrap resamples $\abs{B}$ tends to infinity and can formally be defined as
\begin{equation}
  \label{eq:bs-disappointment}
  D^{\infty}_{\train[n]}[\Es{D_{\bootstrap[n]}}{L(z_{\train[n]}(\obs), y)|x=\obs}\geq c_{n}(z_{\train[n]}(\obs), D_{\train[n]}, \obs)\}]
\end{equation}
as our metric of out-of-sample performance. That is, the bootstrap disappointment probability is the probability that the random predicted cost of our decision as measured by a particular estimator on bootstrap data breaks the original budget for our decision $z_{\train[n]}$. 
We will present in this paper budget minimization formulations for which the  bootstrap disappointment can be controlled explicitly. To do so we introduce our notion of the bootstrap robust counterpart of an estimate-first-predict-later formulation.

\begin{definition}[Bootstrap Robust Counterpart]
  \label{def:bootstrap-robust}
  The budget function $c_n$ and its associated prescriptor $z_{\train[n]}(\obs)\in\arg \min_z \, c_n(z, D_{\train},\obs)$ are said to be the bootstrap robust counterpart with disappointment $b\in [0, 1)$ of an estimate-first-optimize-later formulation with estimator $\Es{D_{\train[n]}}{L(z_{\train[n]}(\obs), y)|x=\obs}$ if we have 
\begin{equation}
  \label{eq:bootstrap-confidence-sla}
  D_{\train[n]}^\infty \left[\Es{D_{\bootstrap[n]}}{L(z_{\train[n]}(\obs), y)|x=\obs} > c_n(z_{\train[n]}(\obs), D_{\train[n]}, \obs)\right] \leq b.
\end{equation}
\end{definition}

Notice however that the bootstrap out-of-sample metric and counterpart as described before are entirely descriptive and do not suggest how to ensure a budget minimization approach does in fact enjoy such good performance. In Section \ref{sec:bootstrap-performance} we will present a disciplined approach to do so based on a distributionally robust approach.

The careful reader may remark that data-driven prescription formulations which are bootstrap robust do \textit{not} promise any theoretical performance guarantees on actual out-of-sample data. Indeed, bootstrap robustness is identified with small disappointment on bootstrap data as opposed to actual out-of-sample data as we have illustrated in Figure \ref{fig:motivational_example} in the context of the newsvendor problem. That bootstrap data serves as a sensible proxy to actual out-of-sample data was advanced already by \citet{efron1982jackknife}. Hence, one may hope that good performance on bootstrap data will translate to good performance on out-of-sample data as well. 

One could try to aim for stronger theoretical guarantees regarding the out-of-sample performance of our proposed robust formulations in the spirit of Equation (\ref{eq:nk-oos}) as done by \citet{rudin2014big, bertsimas2014predictive}. However, such theoretical guarantees require parametric assumptions on the prescription function, i.e., $z^\star\in \mc C$. Such assumptions are hard to justify in the context of many practical problems.
Even for those problems where such assumptions can be justified, it is unclear whether the Rademacher complexity of a general parametric class of functions $\mc C$ can be determined numerically or whether the provided guarantee is in fact not overly pessimistic. 
Instead, we will illustrate in Section \ref{sec:numerical-examples} the empirical efficacy of our bootstrap robust formulations in the context of a small newsvendor problem.

\subsection{Main Contributions}
\label{ssec:main-contributions}

The three main contributions presented in this paper are as follows.

\begin{enumerate}
\item We show that both the classical Nadaraya-Watson and nearest neighbors formulations admit tractable robust counterparts as discussed in Equation (\ref{eq:dro}) for a large variety of ambiguity sets $\mc A \subseteq \mc D_n$. In fact we will do so by treating either formulation as a special case of a more general variable metric estimation formulation. We prove in Section \ref{sec:robust-prescriptive-analytics} that the resulting budget minimization formulations are computationally as tractable as their nominal counterparts. When for instance the estimate-first-optimize-later formulation reduces to a tractable convex optimization problem then so will its robust counterpart. The previous crucial observation testifies to the practical viability of our robust budget minimization formulations.

  Robust Nadaraya-Watson formulations have been proposed before, most notably by \citet{hanasusanto2013robust} in the context of robust dynamic programming. However, we point out in Remark \ref{rem:hanasusanto} that our robust Nadaraya-Watson formulation is of a different nature. Hence, even in the particular context of the Nadaraya-Watson formulation, our application of a distributional robust perspective for prescription problems with contextual information is novel.

\item One particular distributionally robust counterpart based on the relative entropy ball is proven to safeguard against overfitting on bootstrap data as discussed in Definition \ref{def:bootstrap-robust}. We derive in Section \ref{sec:bootstrap-performance} practical finite sample bootstrap performance guarantees regarding the resulting robust \acl{sla}. For this particular bootstrap robust counterpart we further derive a more explicit tractable reformulation based on convex duality. Although we do not prove that our bootstrap robust formulations enjoy guaranteed out-of-sample performance, we demonstrate empirically that this is nevertheless the case in the context of a small newsvendor problem discussed in Section \ref{sec:numerical-examples}.
  
\item Finally, we published a \texttt{Julia} implementation of the algorithms in this work at \url{https://gitlab.com/vanparys/BootstrapRobustAnalytics.jl}.
\end{enumerate}

\section{Estimate-First-Optimize-Later Prescriptions}
\label{sec:prescriptive-analytics}

\acresetall

The estimate-first-optimize-later formulations which we defined in Equation \eqref{eq:eto} are distinct only in so far they are based on  different cost estimates. A large variety of methods in machine learning can provide such estimates. We will focus here solely on the Nadaraya-Watson formulation of \citet{hannah2010nonparametric, rudin2014big} and the nearest neighbors formulation  introduced by \citet{bertsimas2014predictive}. In fact, we will treat here both formulations simultaneously by considering a generalization based on the estimator discussed in \citet{sain2002multivariate}. As this estimator is a local learning method it is based on the concept of neighborhoods depicted visually in Figure \ref{fig:nearest-neighbors}.

\begin{figure}
  \centering
  \includegraphics[width=0.45\textwidth]{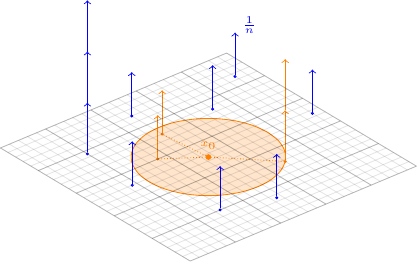}
  \caption{The three and four nearest neighbors (in orange) of the context of interest $\obs$. We depict the neighborhood set $N^3_n(\obs)$ as the orange circle. This neighborhood contains both the three and four nearest neighbors around $\obs$ as the most distant nearest neighbor was seen twice in the training data.}
  \label{fig:nearest-neighbors}
\end{figure}

\begin{definition}[Distance \& Neighborhood]
  Let $\dist(\bar d, \obs)\geq 0$ for all $\bar d\in\Omega$. Assume this distance function enjoys the discrimination property $\dist(\bar d, \bar x) = \dist(\bar d', \bar x) \iff \bar d= \bar d'$ for all $\bar d$, $\bar d'$ in $\Omega$. Let
\[
  \begin{array}{rl}
    & N^j_n(\obs)  \defn \set{\bar d \in \train[n]}{\dist(\bar d, \obs)\leq R^j_{n}}, ~\rm{with}~ R^j_{n} \defn \min\, \set{R\geq 0}{\abs{\set{\bar d \in \train[n]}{\dist(\bar d, \obs)\leq R}}\geq j},
  \end{array}
\]
each containing those $j$ distinct points in the training data closest to our context of interest $\obs$.
\end{definition}

The particular distance $\dist(\bar d, \obs)$ should ideally reflect how relevant an observation $\bar d$ is to our context of interest $x=\obs$. A common choice\footnote{Notice that this distance function does not possess the discrimination property. Indeed, when $\norm{\bar x_i-\obs}_2=\norm{\bar x_j-\obs}$ for $\bar x_i\neq \bar x_j$ we have a tie. The lack of discrimination property translates in ambiguously defined neighborhood sets. The discrimination property may be recovered by deterministically breaking ties based for instance on the value of $\bar y$. \citet{gyorfi2006distribution} propose a randomized alternative by augmenting the covariates with an independent auxiliary uniformly distributed random variable on $[0,1]$. They prove that by doing so ties occur with probability zero.} is to define the distance as a monotonically decreasing function of the Euclidean distance  $\norm{\bar x-\obs}_2$ between the covariates.
\begin{definition}[Estimation formulation]
  \label{def:balloon:nom}
  Our estimation formulation minimizes the weighted average
  \begin{equation}
    \label{eq:balloon:nom}
    z_{\train[n]}(\obs)\in \arg\min_z \,\left\{\Es{D_{\train[n]}}{L(z, y)|x=\obs} \defn \frac{\E{D_{\train[n]}}{ L(z, y)\cdot w_n(x,\obs)\cdot \one\{(x,y)\in N^k_n(\obs)\}}}{\E{D_{\train[n]}}{w_n(x,\obs)\cdot \one\{(x,y)\in N^k_n(\obs)\}}}\right\}
  \end{equation}
  using a positive weighing function $w_n$ over the smallest data neighborhood around the context  $x=\obs$ containing no less than $k$ out of $n$ observations.
\end{definition}

Our estimator hence considers only data within the smallest neighborhood around its context of interest $x=\obs$ containing no less than $k$ out of $n$ observations and is insensitive to all other data points. Within the neighborhood $N^k_{n}(\obs)$ each data point $(\bar x, \bar y)$ is weighted as $w_n(\bar x, \bar y)\geq 0$ with the help of a weighing function. It can easily be verified that this estimator is insensitive to the order of the data points in the training data set. Hence, the estimator is indeed only a function of the training data through its empirical distribution and the total number of samples $n$. Furthermore Assumption \ref{ass:convex_cost_model}, which we will assume holds throughout the remainder of this paper, guarantees that formulation (\ref{eq:balloon:nom}) is convex.

\begin{assumption}[Loss function]
  \label{ass:convex_cost_model}
  The loss function $L(\bar z, \bar y)$ in $\Re_+ \cup \{+\infty\}$ is convex in $\bar z$ for any $\bar y$ and a measurable function of $\bar y$ for any $\bar z$. 
\end{assumption}

\subsection{Nadaraya-Watson and Nearest Neighbors Learning}
\label{sec:stat-cons}

Formulation \eqref{eq:balloon:nom} has several hyper-parameters which have to be chosen with care; the distance function, the number of neighbors, and the weighing function. The statistical properties of the formulation  depend on how these hyper-parameters are chosen. We do not try to establish here the statistical consistency of the formulation in its greatest generality. Rather, we briefly discuss the statistical consistency of its two most common special cases: (i) Nadaraya-Watson and (ii) nearest neighbors learning.

\begin{figure}
  \begin{subfigure}{0.5\textwidth}
    \centering
    \renewcommand{\arraystretch}{1.3}
    \begin{tabular}{|lr|}
      \hline
      Name  & Smoother $S$ \\
      \hline
      Uniform & $\frac12 \one_{\norm{\Delta x}\leq 1}$ \\
      Epanechnikov & $\frac34 (1-\norm{\Delta x}^2)\one_{\norm{\Delta x}\leq 1}$ \\
      Tricubic & $\frac{70}{81}(1-\norm{\Delta x}^3)^3\one_{\norm{\Delta x}\leq 1}$ \\
      Gaussian & $\exp{(-\norm{\Delta x}^2/2)}/\sqrt{2 \pi}$ \\
      \hline
    \end{tabular}
    \vfill
  \end{subfigure}%
  \begin{subfigure}{0.5\textwidth}
    \centering
    \includegraphics[width=0.8\textwidth]{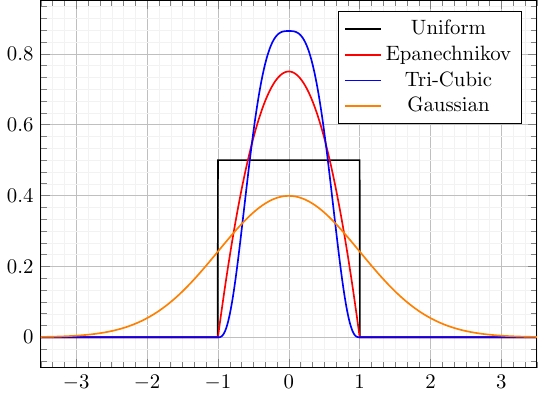}
  \end{subfigure}
  \caption{A comparison of popular common smoother functions $S$.}
  \label{fig:kernels}
\end{figure}

The Nadaraya-Watson formulation of \citet{hannah2010nonparametric, rudin2014big} is identified with the particular choice $k(n)=n$. The Nadaraya-Watson formulation minimizes indeed the weighted average using a positive weighting function $w_n$ over all observations. Typically, the weighing function is taken to be $w_n(\bar x, \obs) = S(\norm{\bar x-\obs}_2/h(n))$ with the help of a positive smoothing function $S$ and bandwidth parameter $h(n)$. Some common popular choices of smoothers are given in Figure \ref{fig:kernels}.  

\begin{mybox}{Example: Nominal Nadaraya-Watson Formulation}
  \begin{equation}
    \label{def:nw}
    z_{\train[n]}(\obs)\in \arg\min_z \,\left\{\Es{D_{\train[n]}}{L(z, y)|x=\obs} \defn \frac{\E{D_{\train[n]}}{ L(z, y)\cdot w_n(x,\obs)}}{\E{D_{\train[n]}}{w_n(x,\obs)}}\right\}.
  \end{equation}
\end{mybox}

Choosing the trivial weighing function $w_n(x, \obs)=1$ on the other hand yields precisely the nearest neighbors formulation already introduced in Section \ref{sec:newsvendor}. The nearest neighbors formulation indeed minimizes the historical average loss
restricted to the smallest data neighborhood around the context $x=\obs$ containing no less than $k$ observations.
\begin{mybox}{Example: Nominal Nearest Neighbors Formulation}
  \begin{equation}
    z_{\train[n]}(\obs)\in \arg\min_z \,\left\{\Es{D_{\train[n]}}{L(z, y)|x=\obs} \defn \frac{\E{D_{\train[n]}}{ L(z, y)\cdot\one\{(x,y)\in N^k_n(\obs)\}}}{\E{D_{\train[n]}}{\one\{(x,y)\in N^k_n(\obs)\}}}\right\}.
  \end{equation}
\end{mybox}
  
Theorems \ref{thm:consistency-nw} and \ref{thm:consistency-nn} in Appendix \ref{sec:unif-cons-estim} guarantee the point-wise consistency of the cost estimates for a fixed decision for Nadaraya-Watson and nearest neighbors learning, respectively. Point-wise consistency of cost estimates crucially does not immediately establish consistency of the associated budget minimization formulations, i.e.,
\begin{equation}
  \label{eq:consistency-nominal}
  \lim_{n\to\infty}\E{D^\star}{L(z_{\data[n]}, y)|x=\obs} = \min_{z}\,\E{D^\star}{L(z, y)|x=\obs}.
\end{equation}
For the latter uniform convergence of the cost estimates needs to be shown. \citet{bertsimas2014predictive} do so under rather mild technical assumptions. Indeed, it suffices to assume that the family of loss functions $\{L(\cdot, \bar y)\}_{\bar y}$ is equicontinuous. This equicontinuity assumption is rather mild as any family of functions with common Lipschitz constant is equicontinuous. Lemma 4 of \citet{bertsimas2014predictive} establishes that under equicontinuity if the cost estimates converge, they do so necessarily uniformly in the decision which in turn guarantees consistency of the associated budget minimization formulation; see also Appendix \ref{sec:unif-cons-estim}.

\subsection{Bootstrap Estimates via Convex Optimization}

We show here that estimation on bootstrap data can be posed as a convex optimization problem. The significance of this result will only become fully clear in Section \ref{sec:robust-prescriptive-analytics}. Recall that any bootstrap data set $\bootstrap[n]$ shares its observations with the training data set $\train[n]$ from which it is resampled modulo their frequency. In terms of its empirical distribution $D_{\bootstrap[n]}$ this translates to
\[
  D_{\bootstrap[n]} \in \mc D_{n,n}\defn\set{D}{\textstyle \sum_{(\bar x, \bar y)\in \Omega_n} D[\bar x, \bar y] =1, ~ n\cdot D[\bar x, \bar y]\in \{0,1,2,\dots,n\} \quad \forall (\bar x, \bar y) \in \Omega_n} \subset \mc D_n.
\]
We will have need for a characterization of our estimator $\Es{D}{L(\bar z, y)|x=\obs}$ defined in Equation \eqref{eq:balloon:nom} as an explicit function of bootstrap distributions $D$ in $\mc D_{n,n}$. This can be done by first associating a partial estimator to each of the  neighborhoods sets $N^j_n(\obs),~j\in[n]$ with the help of a linear optimization problem
\begin{equation}
  \label{eq:partial:estimator}
  \begin{array}{rl}
    \Esp{D}{L(\bar z, y)|x=\bar x} \defn  {\displaystyle\sup_{s>0, P}} & \sum_{(\bar x, \bar y)\in N^j_n(\obs)}\, w_n(\bar x, \obs)\cdot L(\bar z, \bar y)\cdot  P[\bar x, \bar y] \\[0.5em]
    \st  & P[\bar x, \bar y] \geq 0,~s\cdot D[\bar x, \bar y] = P[\bar x, \bar y] \quad \forall (\bar x, \bar y)\in \Omega_n,\\[0.5em]
                                                                  & \sum_{(\bar x, \bar y)\in\Omega_n} P[\bar x, \bar y] = s,~ \sum_{(\bar x, \bar y)\in N^j_n(\obs)} w_n(\bar x, \obs) \cdot P[\bar x, \bar y] = 1,\\[0.5em]
                                                                  & \sum_{(\bar x, \bar y)\in N^j_{n}(\obs)} P[\bar x, \bar y] \geq \frac{k}{n}\cdot s, ~\sum_{N^{j\!-\!1}_{n}\! (\obs)} P[\bar x, \bar y] \leq \frac{k-1}{n} \cdot s.
  \end{array}
\end{equation}
Here an optimization variable $P[\bar x, \bar y]$ for each distinct observed data point $(\bar x, \bar y)$ in the training data $\train[n]$ is introduced together with an additional variable $s$. Note that the domain of the partial estimators as a function of the distribution $D$ satisfies
\begin{equation}
  \label{eq:domain}
  \begin{array}{rl}
    & \dom\,\Esp{D}{L(z, y)|x=\bar x} \\[0.5em]
    \subseteq & \mc D_{n}^j\defn \set{D\in\mc D_n\!\!}{\!\!\textstyle \sum_{N^{j\!-\!1}_{n}\! (\obs)} D[\bar x, \bar y] \leq \tfrac{(k-1)}{n} < \tfrac{k}{n} \leq \sum_{(\bar x, \bar y)\in N^j_{n}(\obs)} D[\bar x, \bar y]}.
  \end{array}
\end{equation}
For a distribution $D$ to be in the domain of the partial estimator the constraints in the optimization formulation given in equation \eqref{eq:partial:estimator} must indeed be feasible. In other words, there must exist some $P$ and $s>0$ for which $s\cdot D[\bar x, \bar y] = P[\bar x, \bar y]$ for all $(\bar x, \bar y)\in \Omega_n$. The last two constraints in equation \eqref{eq:partial:estimator} then imply that any such $D\in \mc D_n$ must also be in $\mc D^j_{n}$. We follow here the standard convention that the supremum over an empty set is unbounded from below. Our estimator can now be decomposed as the maximum of each of the partial estimators previously defined. We defer the proof of the following result to Appendix \ref{ssec:proof-opt-char}.

\begin{theorem}[An Equivalent Optimization Characterization]
  \label{thm:nominal-budget}
  The estimator defined in \eqref{eq:balloon:nom} can be characterized as
  \[
    \begin{array}{rl}
      \Es{D}{L(z, y)|x=\obs} & = \left\{
                               \begin{array}{ll}
                                 {\rm{E}}^{n, 1}_D[L(z, y)|x=\obs]  & \rm{if}~ D\in \mc D^1_n ,\\
                                 \vdots & \vdots \\
                                 {\rm{E}}^{n, n}_D[L(z, y)|x=\obs]  & \rm{if}~ D\in \mc D^n_n ,\\
                               \end{array}\right.\\[2.5em]
                             & = \max_{j\in [n]} \, \Esp{D}{L(z, y)|x=\obs}
    \end{array}
  \]
  for $D\in \mc D_{n,n}$.
\end{theorem}

\section{Distributionally Robust Prescriptions with Contextual Information}
\label{sec:robust-prescriptive-analytics}

In this paper we extend the distributionally robust optimization perspective to prescription problems with contextual information. We construct generic robust formulations as suggested in Equation (\ref{eq:dro}) with the help of a distribution distance function. In the next section we will then show that bootstrap robustness guarantees can be obtained by considering a very particular entropic distance function. However, we will concern ourselves in this section only with the  computational tractability of generic robust formulations.

\begin{definition}[Distribution Distance Function]
  \label{def:distance-function}
  A distribution distance function $R$ is a function quantifying the distance between two empirical distributions in $\mc D_n$ enjoying the following properties:
  \begin{enumerate}
  \item[(i)] \makebox[7.0em]{Discrimination:\hfill} $R(D, D')\geq 0$ for all $D$ and $D'$, while $R(D', D)= 0$ if and only if $D'=D$.
  \item[(ii)] \makebox[7.0em]{Convexity:\hfill} $R(D, D')$ is a convex function in $D$ for all fixed $D'$.
  \end{enumerate}
\end{definition}

We define first a generic robust counterpart following  Equation (\ref{eq:dro}) to a nominal formulation  with respect to the ambiguity set $\mc A(D_{\train[n]}) \defn \set{D\in \mc D_n}{R(D,  D_{\train[n]})\leq r}$ consisting of all empirical distributions at distance not exceeding $r$. 

\begin{definition}[Distributionally Robust Budget Formulations]
  \label{def:sla:robust}
  The distributionally robust counterpart to a estimate-now-optimize-later formulation with respect to the distribution distance function $R$ is defined  as
  \begin{equation}
    \label{eq:boostrap-robust}
    \begin{array}{rl}
      z^{\br}_{\train[n]}(\obs) \in \arg\min_z \, c_n(z, D_{\train[n]}, \obs) \defn \sup_D &  \Es{D}{L(z, Y)|x=\obs} \\[0.5em]
      \st & D \in \mc D_n, \\[0.5em]
          & R(D, D_{\train[n]})\leq r.
    \end{array}
  \end{equation}
\end{definition}

Due to the discrimination property of the distribution distance function, the nominal \acl{sla} is recovered when the robustness radius is zero. In that case we are indeed merely robust with respect to the singleton $\set{D}{R(D, D_{\train[n]})\leq 0}=\{D_{\train[n]}\}$.
Using a robust counterpart instead of nominal \aclp{sla} should protect us against making prescriptions which do well on the training data set but tend to disappoint on out-of-sample data. The robust training prescription $z^{\br}_{\train[n]}(\obs)$ indeed does well not on one particular empirical training distribution  $D_{\train[n]}$ but rather uniformly well on all distributions $\set{D}{R(D, D_{\train[n]})\leq r}$ at distance less than $r$. The particular distance function $R$ dictates which distributions are close to the empirical training distribution and consequently should be chosen with care. Many popular choices are discussed in detail by \citet{postek2016computationally}.

We have in Section \ref{sec:prescriptive-analytics} divided the training data into the nested neighborhoods $N^j_n(\obs),~j\in [n]$. Each of these neighborhoods sets contains those data points closest to the context of interest $x=\obs$. We associated with each of these neighborhoods partial estimators $\Esp{D}{L(z, y)|x=\obs}$. Theorem \ref{thm:nominal-budget} showed that our estimator $\Es{D}{L(z, Y)|x=\obs}$ can be characterized as the maximum of these partial estimators. We likewise first define partial robust estimators via the optimization problems
\begin{equation}
  \label{eq:nna:primal}
  \begin{array}{rl}
    c^j_{n}(z, D, \obs) \defn  {\displaystyle\sup_{s>0, P}} & \sum_{(\bar x, \bar y)\in N^j_n(\obs)}\, w_n(\bar x, \obs)\cdot L(z, \bar y)\cdot  P[\bar x, \bar y] \\[0.5em]
    \st  & s\cdot R({P}/{s}, D)\leq s\cdot r,\\[0.5em]
                                                                 & \sum_{(\bar x, \bar y)\in\Omega_n} P[\bar x, \bar y] = s,~ \sum_{(\bar x, \bar y)\in N^j_n(\obs)} w_n(\bar x, \obs) \cdot P[\bar x, \bar y] = 1,\\[0.5em]
                                                                 & \sum_{(\bar x, \bar y)\in N^j_{n}(\obs)} P[\bar x, \bar y] \geq s \cdot \tfrac{k}{n}, ~\sum_{(\bar x, \bar y)\in N^{j\!-\!1}_{n}\! (\obs)} P[\bar x, \bar y] \leq s \cdot \tfrac{(k-1)}{n}.
  \end{array}
\end{equation}

The previous maximization problems characterizing the partial robust estimators are all concave. Its first optimization variable $s$ is one dimensional, while an additional optimization variable $P(\bar x, \bar y)$ is introduced for each distinct training data point in the empirical support $\Omega_n$. Its first constraint is the only nonlinear one and is convex as the perspective function $s\cdot R({ P}/{s}, D)$ is convex jointly in both variables. In the proof of the theorem stated hereafter we show that
$$c^j_{n}(z, D, \obs) = \sup\, \{\Esp{D}{L(z, y)|x=\obs}:D\in \mc D_n, ~R(D, D_{\train[n]})\leq r\}$$
is indeed the robust counterpart of the partial estimators associated with the nested neighborhoods $N^j_n(\obs)$.

\begin{theorem}[Robust Estimation Formulation]
  \label{thm:nn-primal-representation}
  Our robust estimation formulation (\ref{eq:boostrap-robust}) can be reformulated as the convex optimization problem
  \begin{equation}
    \label{eq:nna}
    z^{\br}_{\train[n]}(\obs)\in\arg \min_z \,c_{n} (z, D_{\train[n]}, \obs)=\arg \min_z\max_{j\in[n]}\, c^{j}_{n} (z, D_{\train[n]}, \obs).
  \end{equation}
\end{theorem}

\begin{proof}
  The chain of equalities
  \[
    \begin{array}{r@{}l}
  \set{(s>0, P)}{\exists D\in \mc D_n ~\st~ s\cdot D= P,\,R(D, D_{\train[n]})\leq r} & = \set{(s>0, P)}{R(\tfrac{P}{s}, D_{\train[n]}) \leq r} \\[0.5em]
                                                                      &  =  \set{(s>0, P)}{s\cdot R(\tfrac{P}{s}, D_{\train[n]}) \leq s\cdot r}    \end{array}
  \]
  implies that the robust partial budget function $c^j_{n}(z, D_{\train[n]}, \obs)$ correspond exactly to the robust version $$\sup\, \{\Esp{D}{L(z, y)|x=\obs}:D\in \mc D_n, ~R(D, D_{\train[n]})\leq r\}$$ of the partial estimators. We have from Theorem \ref{thm:nominal-budget} the decomposition $c_{n} (z, D_{\train[n]}, \obs)\defn\textstyle\max_{j\in[n]}\, c^{j}_{n} (z, D_{\train[n]}, \obs)$. The final optimization formulation is convex as it consists of minimizing the maximum of the individually convex partial budget functions $c^{j}_{n} (z, D_{\train[n]}, \obs)$ as per standing Assumption \ref{ass:convex_cost_model}.
\end{proof}

The previous theorem guarantees the computational tractability of our robust estimation formulation and in particular entails as special cases both robust Nadaraya-Watson and nearest neighbors formulations. We conclude this section by stating that for the former formulation a less involved characterization can be obtained. We defer the proof of this result to Appendix \ref{ssec:corollary_nw_formulation}.

\begin{corollary}(Robust Nadaraya-Watson Formulation)
  \label{cor:nw_formulation}
  The robust Nadaraya-Watson learning formulation ($k(n)=n$) can be reformulated as the convex optimization problem
  \(
  z^{\br}_{\train[n]}(\obs)\in\arg \min_z \,c_{n} (z, D_{\train[n]}, \obs)
  \)
  with
  \[
    \begin{array}{rrl}
      c_{n}(z, D_{\train[n]}, \obs)= & {\displaystyle\sup_{s>0, P}} & \sum_{(\bar x, \bar y)\in \Omega_n}\, w_n(\bar x, \obs)\cdot L(z, \bar y)\cdot  P[\bar x, \bar y] \\[0.5em] 
       & \st  & s\cdot R({P}/{s}, D_{\train[n]})\leq s\cdot r,\\[0.5em]
       &      & \sum_{(\bar x, \bar y)\in\Omega_n} P[\bar x, \bar y] = s,~ \sum_{(\bar x, \bar y)\in \Omega_n} w_n(\bar x, \obs) \cdot P[\bar x, \bar y] = 1.
    \end{array}
  \]
\end{corollary}

\begin{remark}
  \label{rem:hanasusanto}
  Robust Nadaraya-Watson formulations have been proposed before, most notably by \citet{hanasusanto2013robust} in the context of robust dynamic programming. We briefly take here the opportunity to point out that our robust Nadaraya-Watson formulation is different. The robust Nadaraya-Watson formulations found in \citet{hanasusanto2013robust} correspond directly to the more restricted alternative
  \[
    \begin{array}{rl}
      c'_{n}(z, D_{\train[n]}, \obs)\defn  {\displaystyle \sup_D} \,& \tfrac{\sum_{(\bar x, \bar y)\in \Omega_n} w_{n}(\bar x, \obs)\cdot L(z, \bar y)\cdot D[\bar x, \bar y]}{\sum_{(\bar x, \bar y)\in \Omega_n} w_{n}(\bar x, \obs)\cdot D[\bar x, \bar y]}\\[0.5em]
      \st         & R(D, D_{\train[n]})\leq  r, \\[0.5em]
                                                                    &  \sum_{(\bar x, \bar y)\in \Omega_n} w_{n}(\bar x, \obs)\cdot D[\bar x, \bar y] = \sum_{(\bar x, \bar y)\in \Omega_n} w_{n}(\bar x, \obs)\cdot D_{\train[n]}[\bar x, \bar y].
    \end{array}
  \]
  In \citet{hanasusanto2013robust} one particular convex distribution distance function $R$ based on a scaled version of the Pearson distance is singled out. In terms of the convex reformulation given in Corollary \ref{cor:nw_formulation}, this alternative can be seen to correspond to simply restricting the optimization variable $s>0$ to take the particular value $s=\tfrac{1}{\sum_{(\bar x, \bar y)\in \Omega_n} w_{n}(\bar x, \obs)\cdot D_{\train[n]}[\bar x, \bar y]}$. Even in the particular context of the Nadaraya-Watson formulation, our notion of distributional robustness hence seems to be novel.
\end{remark}

\section{Bootstrap Robust Prescriptions with Contextual Information}
\label{sec:bootstrap-performance}

In this section we indicate that for a judiciously chosen distribution distance function our robust prescriptive formulations can be guaranteed to be bootstrap robustness as defined in Definition \ref{def:bootstrap-robust}.

\begin{definition}[The Bootstrap Distance Function]
  \label{def:kl-divergence}
  For two empirical distributions $D$ and $D'$ in $\mc D_n$ we define their bootstrap distance as
  \begin{equation}
    \label{eq:bootstrap-function}
    \B{D}{D'} \defn \textstyle\sum_{(\bar x, \bar y)\in\Omega_n} D[\bar x, \bar y]\cdot \log\left(\frac{D\,\,[\bar x, \bar y]}{D'[\bar x, \bar y]}\right).
  \end{equation}
\end{definition}

The relative entropy is also known as information for discrimination, cross-entropy, information gain or Kullback-Leibler divergence \citep{kullback1951information}. We first prove that its associated robust estimation formulations are statistically consistent under mild conditions.

\subsection{Statistical Consistency}
\label{ssec:consistency_robust}

The bootstrap robust Nadaraya-Watson formulation will be consistent for bounded loss functions when the robustness radius $r(n)$ is appropriately scaled with respect to the bandwidth parameter $h(n)$:

\begin{theorem}[Bootstrap Robust Nadaraya-Watson Formulation]
  \label{thm:consitencty-robust-nw}
  Assume a bounded loss function $L(\bar z, \bar y)<\bar L < \infty$ for all feasible decisions $\bar z$ and parameters $\bar y$. Let $h(n)={cn^{-\delta}}$ for some $c>0$ and $\delta\in(0, \tfrac{1}{\dim(x)})$. Let $S$ be any of the smoother functions listed in Figure \ref{fig:kernels} and the weighing function is taken to be $w_n(\bar x, \obs) = S(\norm{\bar x-\obs}_2/h(n))$. Let the robustness radius satisfy $\lim_{n\to\infty} \tfrac{\sqrt{r(n)}}{h(n)^{\dim(x)}}=0$ for distribution distance function $R=B$. Then, bootstrap robust estimation (with $k(n)=n$) is asymptotically consistent for any $D^\star$, i.e., with probability one
  \[
    \lim_{n\to\infty} ~\E{D^\star}{L(z_{\data[n]}^{r(n)}, y)|x=\obs} = \min_z \,\E{D^\star}{L(z, y)|x=\obs}.
  \]
\end{theorem}

We defer the proof of the previous result to Appendix \ref{ssec:consitencty-robust-nw}. A similar result established in Appendix \ref{ssec:consitencty-robust-nn} holds for the bootstrap robust nearest neighbors formulation when the robustness radius $r(n)$ is scaled appropriately as a function of the number of nearest neighbors $k(n)$:

\begin{theorem}[Bootstrap Robust Nearest Neighbors Formulation]
  \label{thm:consitencty-robust-nn}
  Assume a bounded loss function $L(z, \bar y)<\bar L<\infty $ for all feasible decisions $\bar z$ and parameters $\bar y$. Let $\dist(\bar d=(\bar x, \bar y), \obs) = \norm{\bar x-\obs}_2$ and follow the randomized tie breaking rule discussed in \citet{gyorfi2006distribution}. Let $k(n)=\ceil{\min\{cn^{\delta}, n\}}$ for some $c>0$ and $\delta\in(0,1)$. Let the robustness radius satisfy $\lim_{n\to\infty} \tfrac{n\sqrt{r(n)}}{k(n)}=0$ for distribution distance function $R=B$. Then, bootstrap robust estimation ($w_n(\bar x, \obs)=1$) is asymptotically consistent for any $D^\star$, i.e., with probability one
  \[
    \lim_{n\to\infty} ~\E{D^\star}{L(z_{\data[n]}^{r(n)}, y)|x=\obs} = \min_z \,\E{D^\star}{L(z, y)|x=\obs}.
  \]
\end{theorem}

\subsection{Bootstrap Performance}
\label{ssec:bootstrap_performance}

We will now see that a particular robustness radius is advisable if a certain bootstrap performance is required. To establish that besides consistency, the bootstrap robust estimation formulation suffers only a limited bootstrap disappointment, we will need one elementary result from large deviation theory. The following theorem characterizes the essential large deviation behavior of the empirical distribution $D_{\bootstrap[n]}$ of the bootstrap data resampled from the training data as outlined in Equation~(\ref{data:bootstrapped}).

\begin{figure}
  \centering
  \includegraphics[width=0.45\textwidth]{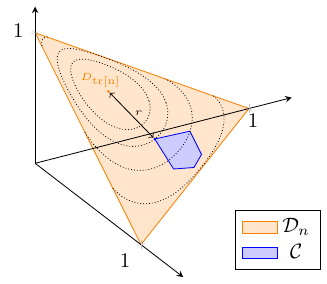}
  \caption{Visualization of the bootstrap inequality \eqref{eq:ldb} stated in Theorem \ref{thm:ldt}. The probability $ D^\infty_{\train[n]}({D}_{\bootstrap[n]}\in \mc C)$ decays with increasing number of samples $n$ at the exponential rate $r\defn\inf_{D\in \mc C}\, \B{D}{D_{\train[n]}}$, which can be interpreted as the bootstrap distance of the empirical training distribution $D_{\train[n]}$ to the set of interest $\mc C$.}
  \label{fig:simplex_ldt}
\end{figure}

\begin{theorem}[The Bootstrap Inequality {\citep[Theorem 1]{csiszar1984sanov}}]
  \label{thm:ldt}
  The probability that the random bootstrap distribution $D_{\bs[n]}$  realizes in any convex set $\mc C \subseteq \mc D_n$ satisfies the finite sample inequality
  \begin{equation}
    \label{eq:ldb}
    D_{\train[n]}^\infty\left[ D_{\bootstrap[n]} \in \mc C \right] \leq \exp\left(- n \cdot \textstyle\inf_{D\in \mc C} \, B(D, D_{\train[n]})\right), \quad \forall n \geq 1.
  \end{equation}
\end{theorem}

The geometry of the bootstrap inequality is visualized in Figure \ref{fig:simplex_ldt}.

Notice that for the optimization problem defining the partial budget functions $c^{j}_{n}$ to be nontrivial for the empirical training distribution $D_{\train[n]}$, the robustness radius $r$ defining our ambiguity set needs to be bigger than the minimum robustness radius
\begin{equation}
  \label{eq:nn-radius}
  \begin{array}{rl}
    r^{j}_n \defn \inf & R(D, D_{\train[n]}) \\[0.5em]
    \st & D \in \mc D^{j}_n
  \end{array}
\end{equation}
where $\mc D^j_n$ was defined in \eqref{eq:domain}.
If this is the case, then the feasible set of the optimization problem \eqref{eq:nna:primal} defining the partial budget cost $c^{j}_{n}$ is indeed non-empty. The extremal bootstrap radii are characterized too as the solution of a tractable convex optimization problem \eqref{eq:nn-radius}. These extremal bootstrap radii will come to  play an important role in the characterization of the bootstrap disappointment suffered by our robust formulation as stated by the following theorem. 

\begin{theorem}[Bootstrap Performance of Robust Formulations]
  \label{thm:performance:nna}
  Our robust formulation with bootstrap distance function ($R=B$) and robustness radius $r$ suffers a bootstrap disappointment defined in \eqref{eq:bootstrap-confidence-sla} at most $$b(n) = \textstyle\sum_{j\in [n]}\exp{(-n\cdot \max \{r, r^j_{n} \})}.$$
\end{theorem}
\begin{proof}
  Let us fix a training data set with empirical distribution $D_{\train[n]}$ and consider a given decision $\bar z$. Let $\bar c^j_{n} \defn \Esp{D_{\train[n]}}{L(\bar z, y)|x=\obs}$ with $\bar c_{n} = \max_{j\in[n]}\bar c^j_{n}$ be the budgeted cost based on the training data. In order to prove the theorem, it suffices to characterize the probability of the event that the empirical distribution $D_{\bootstrap[n]}$ of random bootstrap data resampled from the training data realizes in the set
  \(
  \mc C = \set{D \in \mc D_n}{c_n(\bar z, D, \obs) > \bar c_n} = \cup_{j\in [n]}\,\mc C_j
  \)
  with
  \begin{equation*}
    \begin{array}{rl}
      \mc C_j \defn & \set{D\in\mc D_n^j}{\Esp{D}{L(\bar z, y)|x=\obs}>\bar c_n}\\[1em]
      = & \{D \in \mc D^j_n : \textstyle \sum_{(\bar x, \bar y)\in N^j_n(\obs)} w_n(\bar x, \obs)\cdot L(\bar z, \bar y) \cdot D[\bar x, \bar y] >  \bar c_n  \cdot \sum_{(\bar x, \bar y)\in N_n^j(\obs)} w_n(\bar x,\obs) \cdot D[\bar x, \bar y]\}.
    \end{array}
  \end{equation*}
  Each of the partial sets $\mc C_j$ is a convex polyhedron. We can use the union bound to establish $$D_{\train[n]}^\infty(D_{\bootstrap[n]} \in \mc C) \leq \textstyle\sum_{j\in[n]} D_{\train[n]}^n(D_{\bootstrap[n]} \in \mc C_j).$$ 

The robust budget cost $\bar c_n$ is constructed precisely as to ensure that $\inf_{D\in \mc C_j}\, R(D, D_{\train[n]}) > r$. Indeed, we have the implication $\bar{D} \in \mc C_j \!\implies\! \Esp{\bar D}{L(\bar z, y)|x=\obs} > \bar c_n \geq \sup \,\{ \Esp{D}{L(\bar z, y)|x=\obs} \,|\, R(D, D_{\train[n]})\leq r,\,D\in \mc D_n \}$ which in turn itself implies $R(\bar D, D_{\train[n]}) > r$. By virtue of the inclusion $\mc C_j \subseteq \mc D_n^j$, evidently, we must also have that $\inf_{D\in \mc C_j}\, R(D, D_{\train[n]}) > r^j_{n} \defn \inf \{ R(D, D_{\train[n]}) : D \in \mc D_n^j \}$. Hence, the result follows from the bootstrap inequality \eqref{eq:ldb} applied to each of the probabilities $D_{\train[n]}^\infty(D_{\bootstrap[n]} \in \mc C_j)$ as in this particular case the employed distribution distance function ($R=B$) coincides with the bootstrap distance function.
\end{proof}

The previous theorem guarantees a minimum exponential decay rate for the bootstrap disappointment of any bootstrap robust counterpart formulation, i.e.,
\begin{equation}
  \label{eq:asymptotic_guarantee}
  \limsup_{n\to\infty}\frac{1}{n} \log b(n) \leq -r.
\end{equation}
The bootstrap inequality (\ref{eq:ldb}) is asymptotically exact in the exponential rate. Sanov's Theorem, c.f., \citet[Theorem 6.2.10]{dembo2009large}, indeed establishes that
\begin{equation}
  - \inf_{D \in \interior \mc C} \, B(D, D_{\train}) \leq \liminf_{n\to\infty} \frac{1}{n} \log D_{\train}^\infty\left[ D_{\bootstrap[n]} \in \mc C \right]
  \label{eq:exponential-tightness}
\end{equation}
for any event set $\mc C$ and distribution $D_{\train}$.
Via this lower bound, our bootstrap robust distance function can be shown to the smallest distance function for the guarantee in Equation \eqref{eq:asymptotic_guarantee}, and hence by extension also Theorem \ref{thm:performance:nna}, holds.
The following proposition states indeed that for any other distance function $R$ which is larger than $B$ on some nonempty open set, the guarantee presented Equation \eqref{eq:asymptotic_guarantee} fails to hold in general. The proof of the following result is deferred to Appendix \ref{ssec:proposition-bootstrap-smallest}.

\begin{proposition}
  \label{prop:bootstr-perf-optimal}
  Consider a distance function $R$ with $R(D', D_{\train}) > B(D, D_{\train})=r$ for some $D, D_{\train}\in \mc D_n$ and all $D'$ in some open neighborhood $\mc N\subseteq \mc D_n$ around $D\in \mc N$.
  There exists a nominal formulation so that its robust counterpart given in Equation (\ref{eq:boostrap-robust}) and disappointment set $\mc R=\tset{D\in \mc D_n}{ {\rm E}^n_{D}[L(z^{\br}_{\train}(\obs), y)|x=\obs] > \sup_{R(D', D_{\train})\leq r} {\rm E}^n_{D'}[L(z^{\br}_{\train}(\obs), y)|x=\obs] }$ satisfy
  \[
    -r < \liminf_{n\to\infty} \frac{1}{n} \log D_{\train}^\infty\left[ D_{\bootstrap[n]} \in \mc R \right].
  \]
\end{proposition}

The previous result indicates that our bootstrap distance function $B$ is in some precise sense optimal in safeguarding any nominal formulation against overfitting on bootstrap data.
Choosing the robustness radius $r(n)$ yielding a desired bootstrap disappointment $b$ cannot however be done analytically. However, thanks to the convex characterization \eqref{eq:nn-radius} of the minimum bootstrap radii $r^j_{n}$ it can nevertheless be carried out numerically in a tractable fashion. It is also trivial to see that adding robustness to the extent $r(n) \geq \tfrac{\left(\log(n)+\log(\tfrac{1}{b})\right)}{n}$ suffices to have bootstrap disappointment at most $b$. When we scale the robustness radius in such a way that
\[
  \lim_{n\to\infty} \frac{r(n) \cdot n}{\log(n)} = \infty
\]
then the disappointment on bootstrap data asymptotically converges to zero when the number of training data points $n$ tends to infinity. To be asymptotically consistent the robust nearest neighbors formulation needs to satisfy
\(
\lim_{n\to\infty} \tfrac{n\sqrt{r(n)}}{k(n)} = 0
\)
as pointed out in Theorem \ref{thm:consitencty-robust-nn}.
Asymptotically vanishing disappointment on bootstrap data can be combined with consistency by taking a number of nearest neighbors $k(n)=\ceil{\min\{cn^{\delta}, n\}}$ for some $c>0$ and $\delta\in(0,1)$ while at the same time scaling the robustness radius as $r(n) = t n^\gamma$ with $-1<\gamma<2 (\delta-1)$ for any $t>0$. For the Nadaraya-Watson formulation a slightly improved result can  be established and is stated in Corollary \ref{thm:bootstrap-nw}. We omit its proof and refer the reader to Appendix \ref{ssec:proof-corollary}.

\begin{corollary}[Bootstrap Performance of the Nadaraya-Watson Formulation]
  \label{thm:bootstrap-nw}
  The robust formulation (with $k(n)=n$) with bootstrap distance function ($R=B$) suffers bootstrap disappointment as defined in \eqref{eq:bootstrap-confidence-sla} at most $$b(n) = \exp{(-n\cdot r)}.$$
\end{corollary}

Adding robustness to the extent $r(n) \geq \tfrac{\log(\tfrac  1b)}{n}$ already suffices here to have bootstrap disappointment at most $b$. When we scale the robustness radius in such a way now that
\(
  \lim_{n\to\infty} r(n) \cdot n = \infty
\)
then the disappointment on bootstrap data asymptotically converges to zero when the number of training data points $n$ tends to infinity. To be asymptotically consistent the robust Nadaraya-Watson formulation needs to satisfy
\(
\lim_{n\to\infty} \tfrac{\sqrt{r(n)}}{h(n)^{\dim(x)}} = 0
\)
as pointed out in Theorem \ref{thm:consitencty-robust-nw}. Asymptotically vanishing disappointment on bootstrap data can also here be combined with consistency by scaling the bandwidth parameter $h(n)=cn^{\delta}$ for some $c>0$ and $\delta\in(0,\tfrac{1}{\dim(x)})$ while at the same time scaling the robustness radius as $r(n) = t n^\gamma$ with $-1<\gamma<2 \delta\cdot\dim(x)$ for any $t>0$.

\subsection{Dual Formulations}
\label{ssec:dual_perspectives}

Despite the previous encouraging results regarding the bootstrap performance and consistency of our robust estimation formulation, it is still stated as the solution to a saddle point problem in Equation \eqref{eq:nna}. Both the size and the number of the maximization problems constituting the bootstrap robust formulation grow linearly with the amount of training data samples. The following lemma alleviates one of these concerns by considering a dual formulation of the maximization problems characterizing the partial robust cost functions. We defer its proof to Appendix \ref{ssec:proof-dual}.

\begin{lemma}[Dual representation estimation]
  \label{lemma:nn-dual-representation}
 The partial bootstrap robust ($R=B$) budget $c^j_{n}(\bar z, D, \obs)$ can be represented using a dual convex optimization problem as
  \begin{equation}
    \label{eq:nn:dual}
    \begin{array}{rl}
      \inf  & \alpha \\[0.5em]
      \st   & \alpha\in\Re,~\eta \in\Re_+^2, ~\nu\in\Re_+, \\[0.5em]
            & \nu \log \left(\textstyle\sum_{(\bar x, \bar y)\in N^{j-1}_n(\obs)} \exp({[(L(\bar z, \bar y)-\alpha)\cdot w_n(\bar x,\obs) + \eta_1-\eta_2]}/{\nu})\cdot D[\bar x, \bar y] \right. \\[0.5em]
            & \quad \left. + \sum_{(\bar x, \bar y)\in N^j_n(\obs)\setminus N^{j-1}_n(\obs)} \exp({[(L(\bar z, \bar y)-\alpha)\cdot w_n(\bar x,\obs) + \eta_1]}/{\nu})\cdot D[\bar x, \bar y] \right. \\[0.5em]
            & \quad \left. +  \sum_{\Omega_n\setminus N^j_n(\obs)} D[\bar x, \bar y] \right) + r \cdot \nu - \frac{k}{n}(\eta_1-\eta_2) - \frac{\eta_2}{n}  \leq 0.
    \end{array}
  \end{equation}
  when the robustness radius satisfies $r>r^j_{ n}$ for all $D\in\mc D_n$.
\end{lemma}

The main advantage of using the previous convex dual formulation of the robust budget functions is that finding the optimal prescription $z^{r}_{\train[n]}(\obs)$ now merely requires the solution of a convex optimization problem over the decision $z$ and three additional dual variables $\alpha$, $\beta$, and $\eta$ for each $j\in [n]$. The dependence on the amount of training data is not completely eliminated as the constraint in the dual characterization \eqref{eq:nn:dual} of the partial robust budget $c^j_{n}$ still counts $j$ terms. As the robust nearest neighbors cost function $c_{n}$ consists of the maximum of all of these partial robust cost functions we still have to account for a total number of $O(n^2)$ such terms.

For the particular case of a bootstrap robust nearest neighbors formulation, Lemma \ref{lemma:nn-dual-representation} and Equation \eqref{eq:nna} guarantee the following convex representation.
\begin{mybox}{Example: Bootstrap Robust Nearest-Neighbors Formulation}
  \begin{equation*}
    \begin{array}{r@{\hspace{0.5em}}l}
      z^r_{\train[n]}(\obs)\in \arg\min &\max_{j\in [n]} \alpha_j \\[0.5em]
      \st & z\in \Re^{\dim(z)}, ~\alpha_j\in\Re,~\eta_j \in\Re_+^2, ~\nu_j\in\Re_+ ~~ \forall j\in [n],\\[0.5em]
       & \nu_j \log \left(\textstyle\sum_{(\bar x, \bar y)\in N^{j-1}_n(\obs)} \exp({[L(z,\bar  y)-\alpha_j + \eta_{j, 1}-\eta_{j, 2}]}/{\nu_j})\cdot D_{\train[n]}[\bar x, \bar y] \right. \\[0.5em]
            & \quad \left. + \sum_{(\bar x, \bar y)\in N^j_n(\obs)\setminus N^{j-1}_n(\obs)} \exp({[L( z, \bar y)-\alpha_j + \eta_{j,1}]}/{\nu_j})\cdot D_{\train[n]}[\bar x, \bar y] \right. \\[0.5em]
                                          & \quad \left. +  \sum_{\Omega_n\setminus N^j_n(\obs)} D_{\train[n]}[\bar x, \bar y] \right) + r \cdot \nu_j - \frac{k}{n}(\eta_{j,1}-\eta_{j,2}) - \frac{\eta_{j,2}}{n}  \leq 0 \quad \forall j\in [n].
    \end{array}
  \end{equation*}
\end{mybox}
For the Nadaraya-Watson formulation a more compact optimization formulation can be derived. We defer the proof of the next result to Appendix \ref{ssec:proof-dual-nw}.

\begin{lemma}[Dual Representation Nadaraya-Watson Estimation]
  \label{lemma:nw-dual-representation}
  The bootstrap robust ($R=B$) cost $c_{n}(\bar z, D, \obs)$ for $k(n)=n$ can be represented using a dual convex optimization problem as
  \begin{equation}
    \label{eq:nw:dual}
    \begin{array}{rl}
      \inf  & \alpha \\[0.5em]
      \st   & \alpha\in\Re, ~\nu\in\Re_+, \\[0.5em]
            & \nu \cdot \log \left(\textstyle\sum_{(\bar x, \bar y)\in \Omega_n} \exp\left({(L(\bar z, \bar y)-\alpha)\cdot w_n(\bar x,\obs)}/{\nu}\right)\cdot D[\bar x, \bar y]\right) + r \cdot \nu\leq 0.
    \end{array}
  \end{equation}
  for all $D\in\mc D_n$.
\end{lemma}

For the particular case of a bootstrap robust Nadaraya-Watson formulation, Lemma \ref{lemma:nw-dual-representation} and Equation \eqref{eq:nna} guarantee the following simplified convex representation.

\begin{mybox}{Example: Bootstrap Robust Nadaraya-Watson Formulation}
  \begin{equation*}
    \begin{array}{r@{\hspace{1em}}l}
      z^r_{\train[n]}(\obs)\in \arg\min & \alpha \\[0.5em]
      \st & z\in \Re^{\dim(z)}, ~\alpha\in\Re, ~\nu\in\Re_+, \\[0.5em]
            & \nu \cdot \log \left(\textstyle\sum_{(\bar x, \bar y)\in \Omega_n} \exp\left({(L( z, \bar y)-\alpha)\cdot w_n(\bar x,\obs)}/{\nu}\right)\cdot D_{\train[n]}[\bar x, \bar y]\right) + r \cdot \nu\leq 0.
    \end{array}
  \end{equation*}
\end{mybox}

\section{Numerical Example}
\label{sec:numerical-examples}

We first briefly discuss how our supervised learning formulations were solved and trained in practice. All algorithms were implemented in \texttt{Julia} \citep{bezanson2017julia}.
The nominal Nadaraya-Watson and nearest neighbors formulations of \citet{hannah2010nonparametric,rudin2014big,bertsimas2014predictive} were implemented with the help of the \texttt{Convex.jl} package developed by \citet{udell2014convex}. Taking advantage of the dual representations given in Lemmas \ref{lemma:nw-dual-representation} and \ref{lemma:nn-dual-representation}, the same procedure was followed for their robust counterparts with respect to the bootstrap distance function as well. The corresponding exponential cone optimization problems were solved numerically with the \texttt{ECOS} interior point solver by \citet{domahidi2013ECOS}.

Both the Nadaraya-Watson and nearest neighbors formulations require several hyper parameters such as the smoother function $S$ or the number of neighbors to be learned from data. We considered a Nadaraya-Watson formulation using the Gaussian smoother function given in Figure \ref{fig:kernels}. Likewise, we considered the classical nearest neighbors formulation with the Mahalanobis distance metric $$d((\bar x, \bar y), \obs) = (\bar x-\obs)\tpose \Sigma^{-1}_{\train[n]} (\bar x-\obs)$$ based on the empirical variance and the empirical mean of the covariate data. Potential ties among equidistant points were broken based on the method discussed by \citet{gyorfi2006distribution}. The bandwidth parameter $h(n)$ and the number of nearest neighbors $k(n)$ were determined based on the squared prediction loss performance of the corresponding Nadaraya-Watson or nearest neighbors predictive learner on ten data sets cross validated from the training data.

\subsection{A Newsvendor Problem}
\label{ssec:news-vendor}

We reconsider the newsvendor problem stated in Equation (\ref{prob:news-vendor}). To illustrate the approach we assume that the newsvendor considers both the day of the week $w\in\{\mathrm{Monday}, \dots, \mathrm{Sunday}\}$ and the outside temperature $t \in \Re$ as features which potentially influence customer demand.
We consider synthetic training data drawn as independent samples drawn from the conditional distribution
\[
   D^\star(\obs = ( t_0, w_0)) = N\!\left(100+(t_0-20) + 20 \cdot \mb I(w_0\in\{\mathrm{Weekend}\}), 16\right)
\]
and where the day of the week and outside temperature are independent random variables distributed uniformly and normally as $N(20, 4)$, respectively. 
We shall use this synthetic data newsvendor problem to illustrate the bootstrap disappointment of the robust Nadaraya-Watson and nearest neighbors formulations in a particular context of interest, i.e., $\obs = (t_0, w_0) = (10 ^\circ\mathrm{C},\,\mathrm{Friday})$. We remark that \citet{rudin2014big} discuss a Nadaraya-Watson learning formulation of a similar data-driven newsvendor problem which they denote as a kernel optimization formulation.

We first investigate to what extent our bootstrap robust formulations guard against disappointment on both synthetic bootstrap data and actual out-of-sample validation data. Given an action $z_{\train[n]}^r$ calibrated to the training data set and its budgeted cost estimate $c_n(z_{\train[n]}^r, D_{\train[n]}, \obs)$, we approximate its bootstrap and out-of-sample validation disappointment as suggested in Equations \eqref{eq:bootstrap-approx} and (\ref{eq:val-disappointment}) using a large number $\abs{B}=\abs{V}=20,000$ of resamples. In Figure \ref{fig:disappointment-in-sample}, we present these disappointment probabilities as a function of the number of training samples for the nominal and robust Nadaraya-Watson and nearest neighbors formulations. The nominal Nadaraya-Watson of \citet{hannah2010nonparametric,rudin2014big} and the nearest neighbors formulation by \citet{bertsimas2014predictive} corresponds to the cases depicted with $r=0$. Such nominal formulations do not safeguard against overfitting as they disappoint on both synthetic bootstrap data and actual out-of-sample data about half the time. The dotted lines visualize the upper bounds concerning the bootstrap disappointment of the bootstrap robust Nadaraya-Watson and nearest neighbors formulation given in Corollary \ref{thm:bootstrap-nw} and Theorem \ref{thm:performance:nna}, respectively. The guarantee in case of the nearest neighbors formulation is not as tight as its Nadaraya-Watson counterpart mostly due to the use of the union bound in the proof of Theorem \ref{thm:performance:nna}. Nevertheless, large deviation theory via Equation \eqref{eq:exponential-tightness} ensures that the empirical bootstrap disappointments and their corresponding theoretical upper bound in either formulation drop to zero at the same exponential rate $r$. Finally, the fact that the disappointment probability on actual out-of-sample data trends down at a similar rate with increasing number of samples as the disappointment on synthetic bootstrap data indicates that bootstrap data is here a sensible proxy to actual out-of-sample data.

\begin{figure}
  \centering
  \begin{subfigure}{0.45\textwidth}
    \includegraphics[width=\textwidth]{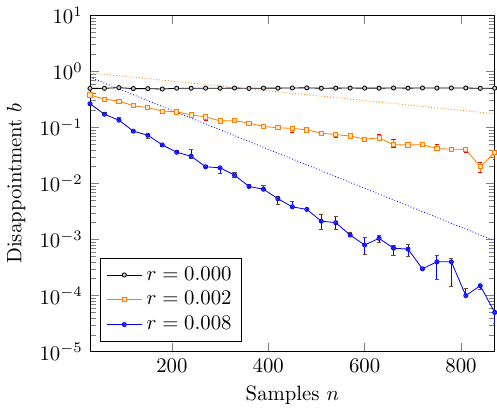}
    \caption{Nadaraya-Watson Formulations}
  \end{subfigure}%
  \begin{subfigure}{0.45\textwidth}
    \includegraphics[width=\textwidth]{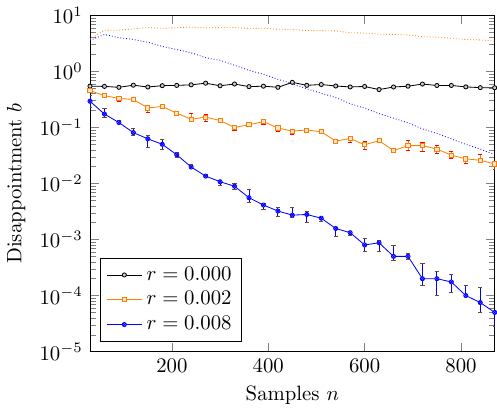}
    \caption{Nearest Neighbors Formulations}
  \end{subfigure}\\
  \begin{subfigure}{0.45\textwidth}
    \includegraphics[width=\textwidth]{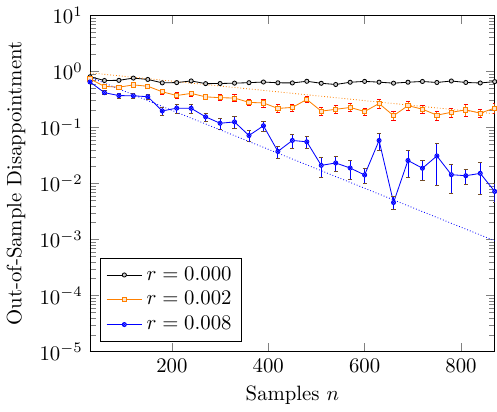}
    \caption{Nadaraya-Watson Formulations}
  \end{subfigure}%
  \begin{subfigure}{0.45\textwidth}
    \includegraphics[width=\textwidth]{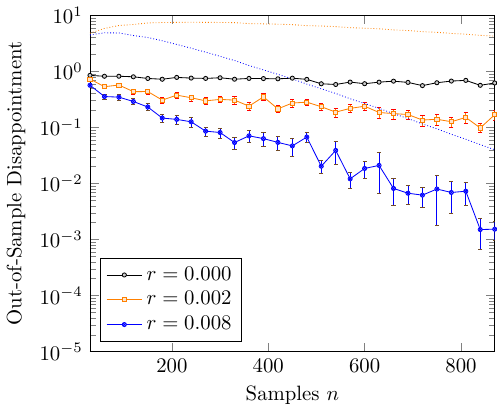}
    \caption{Nearest Neighbors Formulations}
  \end{subfigure}  
  \caption{The empirical bootstrap disappointment $b$ (first row) and the out-of-sample disappointment (second row) of the Nadaraya-Watson (first column) and nearest neighbors (second column) formulations in function of the number of samples $n$. The nominal Nadaraya-Watson and nearest neighbors formulation corresponds to the case $r=0$. Such nominal formulations do not safeguard against overfitting either on bootstrap or actual out-of-sample data as they disappoint on both about half the time. The dotted lines visualize the upper bounds concerning the bootstrap disappointment of the bootstrap robust Nadaraya-Watson and nearest neighbors formulation given in Theorem \ref{thm:bootstrap-nw} and Theorem \ref{thm:performance:nna}, respectively. Large deviation theory \citep{csiszar1984sanov} indicates that these bootstrap upper bounds and the actual bootstrap disappointments of either formulation drop to zero at the same exponential rate $r$. Finally, the fact that the disappointment on actual out-of-sample data trends down at a similar rate with increasing number of samples as the disappointment on synthetic bootstrap data indicates that bootstrap data is here a sensible proxy to actual out-of-sample data.}
  \label{fig:disappointment-in-sample}
\end{figure}

Working with synthetic data also gives us the opportunity to compare the robust and nominal formulations in terms of their actual expected cost and out-of-sample disappointment. Indeed, as the distribution $D^\star$ is known we can quantify the actual cost $\E{D^\star}{L(z_{\data[n]}, y)|x=\obs}$ of any proposed decisions $z_{\data[n]}$ analytically.  To guarantee a fair comparison, the robustness radius used here for the formulation proposed by \citet{hanasusanto2013robust} was chosen so that both formulations matched in terms of their in-sample estimated or budgeted cost $c_{n}(z_{\train[n]}(\obs), D_{\train[n]}, \obs)$ which hence is not reported. Proposition \ref{prop:bootstr-perf-optimal} indicates that our bootstrap distance function should be preferred when guarding formulations against disappointment on bootstrap data over any other distance function which as indicated in Remark \ref{rem:hanasusanto} includes the distance function considered by \citet{hanasusanto2013robust}. However, this advantage is rather small and both formulations enjoy a similar very performance as indicated in Table \ref{tab:hk} in which we report the actual cost $c_{\rm{oos}}$ and disappointment on both bootstrap data $b_{\rm{in}}$ and out-of-sample data $b_{\rm{oos}}$ aggregated over 250 independent training data sets. That being said, the main practical advantage of our bootstrap distance function is that its size $r$ can be calibrated in advance to guarantee a certain bootstrap disappointment (here $b=0.1$) via Theorem \ref{thm:performance:nna} (or Corollary \ref{thm:bootstrap-nw}) whereas the robustness radius of the distance function in the formulation \citet{hanasusanto2013robust} comes with no such interpretation and ordinarily would need to be calibrated using a potentially computationally expensive validation procedure. 

\begin{table}
  \centering
  \begin{tabular}{l@{\hspace{1.5em}}cccccc}
    \cline{1-7}
     & \multicolumn{3}{c}{\bfseries BR (b = 0.1)} &  \multicolumn{3}{c}{\bfseries HK (matched)}\\
    \cmidrule(r{5pt}){2-4}\cmidrule(l{5pt}r{0pt}){5-7} 
     & $b_{\rm{in}}$ & $b_{\rm{oos}}$ & $c_{\rm{oos}}$ & $b_{\rm{in}}$ & $b_{\rm{oos}}$ & $c_{\rm{oos}}$ \\
    \cline{1-7}
    $N= 100$ &  0.0209 & 0.0944 & 26.3145 & 0.0206 & 0.0948 & 26.524\\
    $N= 150$ &  0.0012 &  0.0068 & 21.0093 & 0.0010 & 0.0067 & 20.9988\\
    $N= 200$ &  0.0025 & 0.0126 & 17.2461 & 0.0023 & 0.0126  & 16.9906\\
    $N= 250$ &  0.0159 & 0.0261 & 14.3198 & 0.0142 & 0.0287 & 13.6155\\
    \cline{1-7}
  \end{tabular}
  \caption{Comparison between our bootstrap robust Nadaraya-Watson formulation (BR) and the robust formulation (HK) proposed by \citet{hanasusanto2013robust}.}
  \label{tab:hk}
\end{table}

We are also interested in seeing how each of the methods fares if we augment the covariates with a number $d$ of irrelevant spurious observations generated from independently sampling a standard normal distribution. It should not come as a surprise that when no spurious covariates are introduced the robust approaches in blue outperform their nominal counterparts in orange significantly as can be seen in Figure \ref{fig:dimension}. The reported cost is the average among $20,000$ random training sets each of length $n=200$. Here, the amount of robustification $r$ used in either formulation was determined based on ten fold cross validation. Given only $n=200$ training samples, both the robust Nadaraya-Watson and nearest neighbors formulation do indeed come close to the theoretically minimal cost $\min_z\, \E{D^\star}{L(z, y)|x=\obs}$ visualized as the black line for reference. As more spurious covariates are injected in the training data set, the performance of any data-driven method must evidently degrade. Ultimately as $d$ tends to infinity, the noise completely drowns out the signal. The performance of data-driven methods is expected to tend to the optimal cost $\min_z\, \E{D^\star}{L(z, y)}$ without using any covariate information and is visualized as the red line.

\begin{remark}
  It is perhaps curious to observe that the performance of the nominal Nadaraya-Watson and nearest neighbors decisions behave differently than their robust counterparts. Counter-intuitively, the performance of both nominal formulations initially improves with the introduction of spurious covariates. We observed empirically that when no spurious covariates are introduced the nominal cost estimates $\Es{D_{\train[n]}}{L(\bar z, y)|x=\obs}$ for all $\bar z$ on which the nominal formulations are based tend to overfit the training data. Consequently, the nominal actions suffer poor out-of-sample performance. Our robust formulations do a good job in alleviating this adverse phenomenon without much performance loss. The introduction of artificial noise in the data through spurious covariates seems to provide implicit protection against overfitting to the training data by preventing the selection, at least when using cross-validation, of a small bandwidth parameter $h(n)$ or a small number of neighbors $k(n)$ in the Nadaraya-Watson and nearest neighbor formulation, respectively. The addition of spurious covariates has a similar effect as robustification but crucially does comes with a loss of performance.
\end{remark}

\begin{figure}
  \centering
  \begin{subfigure}{0.45\textwidth}
    \includegraphics[width=\textwidth]{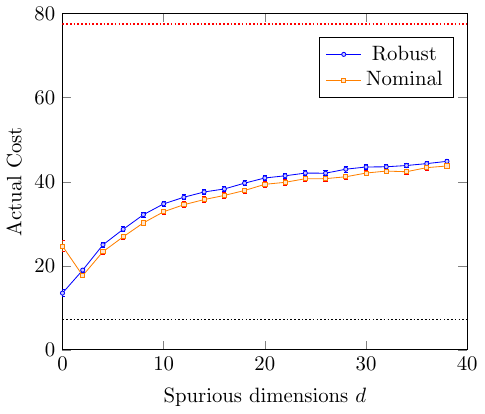}
    \caption{Nadaraya-Watson Formulations}
  \end{subfigure}%
  \begin{subfigure}{0.45\textwidth}
    \includegraphics[width=\textwidth]{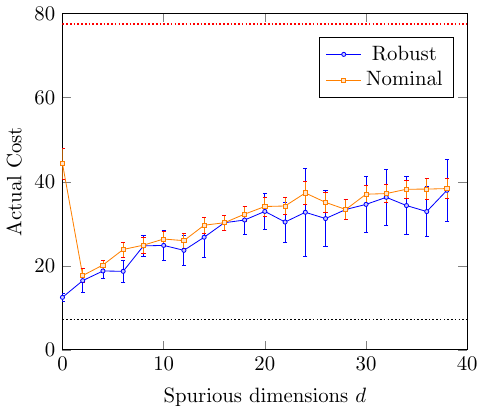}
    \caption{Nearest Neighbors Formulations}
  \end{subfigure}
  \caption{The average actual cost $\E{D^\star}{L(z, y)|x=\obs}$ of the decisions proposed by the nominal and robust Nadaraya-Watson formulation as well as the nominal and robust nearest neighbors formulation. The reported cost is the average among $20,000$ random training sets each of length $n=200$. The dimension $d$ reflects the number of spurious covariates introduced. The black line correspond to the optimal full information cost $\min_z\, \E{D^\star}{L(z, y)|x=\obs}$, while the red line represents the optimal no information cost $\min_z\, \E{D^\star}{L(z, y)}$. The robust formulations dominate with large statistical significance their nominal counterparts when no spurious covariates ($d=0$) are introduced. When spurious covariates are introduced ($d>0$) the cost predictions $\Es{D_{\train[n]}}{L(z, y)|x=\obs}$ on which all formulations are based tend to be under-calibrated leading to both a loss of performance and leave limited room for improvement via robustification. }
  \label{fig:dimension}
\end{figure}

\begin{figure}[ht]
    \centering
  \begin{subfigure}{0.45\textwidth}
    \includegraphics[width=\textwidth]{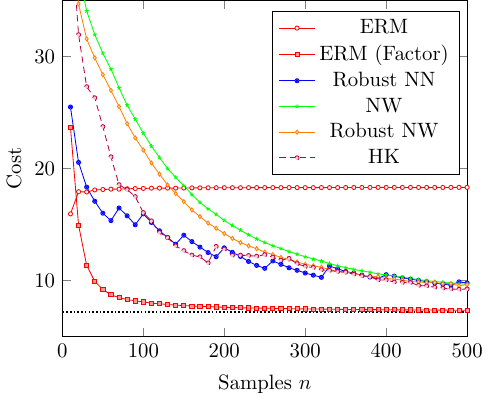}
  \end{subfigure}%
  \begin{subfigure}{0.45\textwidth}
    \includegraphics[width=\textwidth]{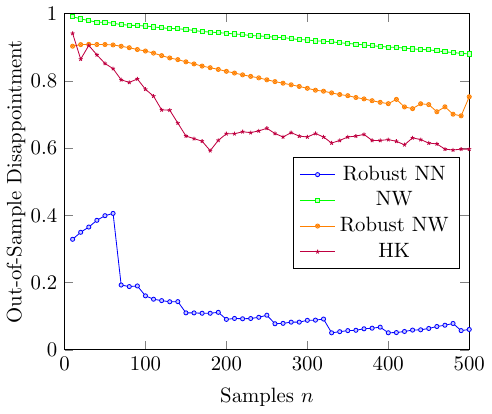}
  \end{subfigure}
  \caption{Comparison with the empirical risk minimization (ERM) and Nadaraya-Watson (NW) formulations proposed in \citet{rudin2014big}. The reported cost $\E{D^\star}{L(z_{\data[n]}, y)|x=\obs}$ is the average among a $1,000$ random data sets. Naive application of the ERM formulation yields a biased learning method which does not achieve the minimum cost (indicated as the dotted black line) even in the large sample limit in contrast to the the nonparametric robust Nadaraya-Watson and nearest neighbors formulations as well as the formulation (HK) proposed by \citet{hanasusanto2013robust}. With the help of a factor model, the ERM formulation can be made to achieve the minimum cost and in fact does so with fewer samples than the nonparametric learning formulations. We also show the out-of-sample disappointment of the three nonparametric approaches.
    We remark here that as empirical risk formulations do not come with a cost estimate for their optimal action in the context of interest they do not come with a similar notion of out-of-sample disappointment.
    Our bootstrap robust nearest neighbors formulation dominates here the other two nonparametric formulations both in terms of expected cost as well as out-of-sample disappointment.
  }
  \label{fig:comparison-erm}
\end{figure}

A supervised data version of this newsvendor problem is discussed by \citet{rudin2014big} based on the empirical risk formulation stated in Equation \eqref{eq:emp_risk} with $\mc C=\mc F_{\lin}$ the set of all linear functions and an estimate-first-optimize-later formulation using a Nadaraya-Watson estimator.
In this synthetic example, the full information solution stated in Equation (\ref{eq:sol-news-vendor}) simplifies to
$$z^\star(\obs=(t_0, w_0))=  100+(t_0-20) + 20 \cdot \mb I(w_0\in\{\mathrm{Weekend}\}) + 4\sqrt{2}\,\erf^{-1}(\tfrac{(b-h)}{(b+h)})$$
where $\erf$ denotes the standard error function. When representing the categorical variables naively as integers, i.e., $\{\mathrm{Monday}, \dots, \mathrm{Sunday}\}=\{1,\dots,7\}$, the function $z^\star(\cdot)\not\in \mc F_{\lin}$ is not linear in this representation and hence as discussed in the introduction, the associated empirical risk minimization formulation is biased.  Consequently, the empirical risk formulation does not converge to the full information solution as empirically verified in Figure \ref{fig:comparison-erm}. In stark contrast, our robust Nadaraya-Watson and nearest neighbors formulations as well as the formulation proposed by \citet{hanasusanto2013robust} do not require such parametric assumptions on the full information solution to be asymptotically consistent. Indeed, as illustrated in Figure \ref{fig:comparison-erm}, all eventually achieve the full information minimum cost when $n$ tends to infinity. In this example, our bootstrap robust nearest neighbors formulation dominates the other nonparametric formulations both in terms of expected cost as well as out-of-sample disappointment.

Clearly, by considering a standard ``dummy'' factor representation \citep{friedman2001elements}[Section 3.2] of the categorical variables the full information solution $z^\star(\cdot)$ becomes linear. Consequently, the empirical risk formulation with such factor model does converge to the full information solution and does so faster than our nonparametric learning formulations; see again Figure \ref{fig:comparison-erm}. This observation is not surprising and very much in line with the the fact that parametric learning is statistically more powerful when its parametric assumptions apply but may fail when its assumptions do not hold. In practice, which method should be favored will hence strongly depend on whether or not sensible parametric assumptions can be assumed on the prescription function $z^\star(\cdot)$ in a particular application. Furthermore, the computational speed of each formulation will in practice also be of importance.
In Figure \ref{fig:computational} we indicate that as each formulation requires merely the solution of a convex optimization problem the computational times do not exceed one second. However, as remarked already by \citet{rudin2014big}, the nonparametric formulations are as a whole computationally less demanding than the empirical risk minimization formulations.

\begin{figure}[ht]
  \centering
  \includegraphics[width=0.45\textwidth]{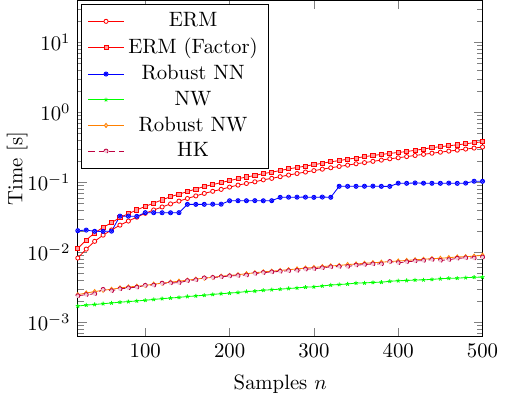}
  \caption{Comparison of all discussed formulations in terms of the time it takes to solve them. As ultimately all formulations require the solution of a convex optimization problem the computational times do not exceed one second. As remarked already by \citet{rudin2014big}, the nonparametric formulations (Robust NN, NW, Robust NW, HK) are as a whole computationally less demanding than the empirical risk formulations (ERM, ERM (Factor)). The difference between our bootstrap robust Nadaraya-Watson formulation and the formulation proposed by \citet{hanasusanto2013robust} in terms of computational times is negligible and barely visible.
  }
  \label{fig:computational}
\end{figure}

\section*{Acknowledgments}

The authors would like to thank the reviewers for their careful reading of this manuscript which greatly improved its overall exposition.
The second author is generously supported by the Early Post.Mobility fellowship No.\ 165226 of the Swiss National Science Foundation.

\setlength\bibitemsep{10pt}
\printbibliography

\appendix

\section{Uniform Consistent Estimators}
\label{sec:unif-cons-estim}

Nadaraya-Watson estimation can be shown to be point-wise consistent when using an appropriately scaled bandwidth parameter $h(n)$ for any of the smoother functions listed in Figure \ref{fig:kernels}:
  
\begin{theorem}[{\citet{walk2010strong}}]
  \label{thm:consistency-nw}
  Let us have a loss function satisfying $\E{D^\star}{\abs{L(\bar z, y)} \cdot \max\{\log(\abs{L(\bar z, y)}),0\}} < \infty$ for all $\bar z$. Let the bandwidth $h(n)={cn^{-\delta}}$ for some $c>0$ and $\delta\in(0, \tfrac{1}{\dim(x)})$. Let $S$ be any of the smoother functions listed in Figure \ref{fig:kernels} and the weighing function is taken to be $w_n(x, \obs) = S(\norm{x-\obs}_2/h(n))$. Then, our estimation formulation (with $k(n)=n$) is asymptotically consistent for any $D^\star$, i.e., with probability one we have
  \[
    \lim_{n\to\infty} \Es{D_{\data[n]}}{L(\bar z, y)|x=\obs} = \E{D^\star}{L(\bar z, y) | x=\obs} \qquad \forall \bar z.
  \]
\end{theorem}

Nearest neighbors estimation is also consistent under very mild technical conditions provided that the number of neighbors $k(n)$ is scaled appropriately with the number of training data samples:

\begin{theorem}[{\citet{walk2010strong}}]
  \label{thm:consistency-nn}
  Assume $\dist(\bar d=(\bar x, \bar y), \obs) = \norm{\bar x-\obs}_2$ and follow the random tie breaking rule discussed in \citet{gyorfi2006distribution}. Let $k(n)=\ceil{\min\{cn^{\delta}, n\}}$ for some $c>0$ and $\delta\in(0,1)$. Then, our estimation formulation (with $w_n(\bar x, \obs)=1$) is asymptotically consistent for any $D^\star$, i.e., with probability one we have
  \[
    \lim_{n\to\infty} \Es{D_{\data[n]}}{L(\bar z, y)|x=\obs} = \E{D^\star}{L(\bar z, y) | x=\obs} \quad \forall \bar z.
  \]  
\end{theorem}

An estimator is denoted as uniformly consistent if the event
\begin{equation}
  \label{eq:uniform-consistency}
  \lim_{n\to\infty}D^{\star\infty}\left[\max_{\bar z}\, |\Es{D_{\data[n]}}{L(\bar z, y)|x=\obs} - \E{D^\star}{L(\bar z, y) | x=\obs}|\leq \epsilon\right]=1
\end{equation}
for any $\epsilon>0$ and $D^\star$. That is, the estimator predicts the cost for all potential decisions well to any arbitrary accuracy when only given access to a sufficiently large amount of training data. It is not difficult to see that this uniform consistency in turn implies the consistency of the budget minimization formulation. Indeed, with probability one we have
\begin{align}  
    & \textstyle\lim_{n\to\infty}D^{\star\infty}\left[\min_{z}\,\E{D^\star}{L(z, y)|x=\obs}+\epsilon \geq \Es{D_{\data[n]}}{L(z^\star(\obs), y)|x=\obs}\right]=1 \label{eq:consistency-1}\\[0.5em]
    & \Es{D_{\data[n]}}{L(z^\star(\obs), y)|x=\obs} \geq \Es{D_{\data[n]}}{L(z_{\data[n]}, y)|x=\obs} \label{eq:consistency-2}\\[0.5em]
    & \textstyle\lim_{n\to\infty}D^{\star\infty}\left[\Es{D_{\data[n]}}{L(z_{\data[n]}, y)|x=\obs} \geq \E{D^\star}{L(z_{\data[n]}, y)|x=\obs}-\epsilon\right]=1 \label{eq:consistency-3}
\end{align}
Here, the limit (\ref{eq:consistency-1}) follows from the consistency of the estimator with regards to the full information decision $z^\star(\obs)$. Inequality (\ref{eq:consistency-2}) is a direct consequence of the characterization of the data-driven decision $z_{\data[n]}$ as a minimizer of $\Es{D_{\train[n]}}{L(z, y)|x=\obs}$. Remarking that $\max_{\bar z}\, |\Es{D_{\data[n]}}{L(\bar z, y)|x=\obs} - \E{D^\star}{L(\bar z, y) | x=\obs}| \leq \epsilon \implies \Es{D_{\data[n]}}{L(z_{\data[n]}, y)|x=\obs} \geq \E{D^\star}{L(z_{\data[n]}, y)|x=\obs}-\epsilon$ the limit  (\ref{eq:consistency-3}) follows from the uniform consistency of the estimator stated in Equation (\ref{eq:uniform-consistency}). Chaining the inequalities in results (\ref{eq:consistency-1}), (\ref{eq:consistency-2}) and (\ref{eq:consistency-3}) implies that
\[
  \textstyle\lim_{n\to\infty}D^{\star\infty}\left[ 0 \leq \E{D^\star}{L(z_{\data[n]}, y)|x=\obs} - \min_{z}\,\E{D^\star}{L(z, y)|x=\obs} \leq 2\epsilon\right]=1.
\]
Here the optimality gap between the cost of the full information decision $z^\star(\obs)$ and the cost of the data-driven decision $z_{\data[n]}$ is bounded by $2\epsilon$. As this result holds for any arbitrary $\epsilon>0$ the data-driven formulation is asymptotically consistent.

\section{Proofs}

\subsection{Proof of Theorem \ref{thm:nominal-budget}}
\label{ssec:proof-opt-char}

\begin{proof}
  First note that the domain of the partial estimators as a function of the distribution $D$ satisfies
  $$\rm{dom}~\Esp{D}{L(z, Y)|x=\obs} \subseteq \mc D_{n}^j.$$
  For a distribution $D$ to be in the domain of the partial estimator the constraints in equation \eqref{eq:partial:estimator} must indeed be feasible. In other words, there must exists some $P$ and $s>0$ for which $s\cdot D[\bar x, \bar y] = P[\bar x, \bar y]$ for all $(\bar x, \bar y)\in \Omega_n$. The last two constraints in equation \eqref{eq:partial:estimator} then imply that any such $D\in \mc D_n$ must also be in $\mc D^j_{n}$.  

  Any $D\in \mc D_{n,n}$ is the empirical distribution of some bootstrap data set, say $\bs[n]$, consisting of $n$ observations from the training data set. Each set $\mc D_{n,n}^j\defn \mc D_n^j\cap\mc D_{n,n}$ has a very natural interpretation in terms of $N^j_n(\obs)$ being the smallest neighborhood containing at least $k$ observations of this associated bootstrap data set $\bs[n]$. Indeed, $D\in\mc D_{n,n}^j$  is in terms of the associated data set equivalent to 
  \[
    \begin{array}{llll}
     k & \leq & \sum_{(\bar x, \bar y)\in\bs[n]} \mb 1 \{(\bar x, \bar y)\in N^{j}_{n}(\obs)\}  &= n \cdot \sum_{(\bar x, \bar y)\in N^j_n(\obs)} D[\bar x, \bar y], \\[0.5em]
     k & > & \sum_{(\bar x, \bar y)\in\bs[n]}\mb 1\{(\bar x, \bar y)\in N^{j\!-\!1}_{n}\! (\obs)\} &=  n\cdot \sum_{(\bar x, \bar y)\in N^{j\!-\!1}_{n}\! (\obs)} D[\bar x, \bar y].
    \end{array}
  \]
  The first inequality implies that the neighborhood $N^j_n(\obs)$ contains at least $k$ observations of the associated data set. Note that the sets $N^j_n(\obs)$ are increasing with increasing $j$ in terms of set inclusion. The latter inequality hence implies that the biggest smaller neighborhood $N^{j\!-\!1}_n(\obs)$ does not contain $k$ observations. Both conditions taken together thus imply that $N^j_n(\obs)$ is the smallest neighborhood which contains at least $k$ samples of the bootstrap data set. Any $D$ in $\mc D_{n,n}$ is an element in one and only one set $\mc D^j_{n,n}$ as the smallest neighborhood  containing at least $k$ samples is uniquely defined for any data set. Formally, $\mc D^j_{n,n}\cap\mc D^{j'}_{n,n}=\emptyset$ for all $j\neq j'$ and moreover $\cup_{j\in[n]} \mc D_{n,n}^j=\mc D_{n,n}$. Notice also that the only feasible $s$ in the constraints defining the partial predictors in equation \eqref{eq:partial:estimator} is the particular choice $s = \tfrac{1}{\sum_{(\bar x, \bar y)\in N^j_n(\obs)} w_n(\bar x, \obs) \cdot P[\bar x, \bar y]}>0$. Hence, we must have that for any $D\in\mc D^j_{n,n}$ the partial estimator equates to $$\Esp{D}{L(\bar z, y)|x=\obs} = \frac{\textstyle\sum_{(\bar x,\bar y)\in N^j_{n}(\obs)} L(\bar z, \bar y) \cdot w_n(\bar x,\obs)\cdot D(\bar x, \bar y)}{\textstyle\sum_{(\bar x,\bar y)\in N_n^j(\obs)} w_n(\bar x,\obs) \cdot D(\bar x, \bar y)} \qquad \forall \bar z$$ which is precisely the weighted average over the neighborhood $N^j_n(\obs)$. We have hence for all $\bar z$ that 
  \begin{equation*}
    \begin{array}{r@{}l}
      \max_{j\in [1,\dots,n]} \, \Esp{D}{L(\bar z, y)|x=\obs} ~& = \left\{
                                                           \begin{array}{ll}
                                                             \frac{\textstyle\sum_{(\bar x,\bar y)\in N^1_{n}(\obs)} L(\bar z, \bar y) \cdot w_n(\bar x,\obs)\cdot D[\bar x, \bar y]}{\textstyle\sum_{(\bar x,\bar y)\in N_n^1(\obs)} w_n(\bar x,\obs) \cdot D[\bar x, \bar y]} & \rm{for~} D\in \mc D^1_n ,\\
                                                             \vdots &  \vdots \\
                                                             \frac{\textstyle\sum_{(\bar x,\bar y)\in N^n_{n}(\obs)} L(\bar z, \bar y) \cdot w_n(\bar x,\obs)\cdot D(\bar x, \bar y)}{\textstyle\sum_{(\bar x,\bar y)\in N_n^n(\obs)} w_n(\bar x,\obs) \cdot D(\bar x, \bar y)} & \rm{for~} D\in \mc D^n_n.
                                                           \end{array}\right.\\[4em]
                                                         & = \Es{D}{L(\bar z, y)|x=\obs}
    \end{array}
  \end{equation*}
  as we have argued that $D\in \mc D^j_{n,n}$ if and only if $N^j_n(\obs)$ is the smallest neighborhood containing at least $k$ data points. As the empirical distribution $D$ and associated data set $\bs[n]$ were chosen arbitrary the result follows.
\end{proof}

\subsection{Proof of Proposition \ref{prop:bootstr-perf-optimal}}
\label{ssec:proposition-bootstrap-smallest}

\begin{proof}
  Observe that $\set{D \in \mc D_n}{R(D, D_{\train})\leq r}\,\cap\, \mc N=\emptyset$ by construction. Furthermore, $\set{D \in \mc D_n}{R(D, D_{\train})\leq r}$ is convex as $R$ is convex in its first argument. Without loss of generality, the neighborhood $\mc N$ is convex as well.
   By the Hahn–Banach separation Theorem an open convex set can be separated linearly from any other disjoint convex set. That is, there must exist a function $G$ and constant $a\in \Re$ so that
  \[
    \E{D}{G(x, y)} \leq a < \E{D'}{G(x, y)}   \quad \forall D\in \set{D \in \mc D_n}{R(D, D_{\train})\leq r}, ~D'\in \mc N.
  \]
  Hence, $\interior\mc N  = \mc N \subseteq \mc R' \defn \tset{D\in \mc D_n}{\E{D}{G(x, y)} >a \geq \sup_{R(D', D_{\train})\leq r} \E{D'}{G(x, y)}} \subseteq \mc R $ with $\mc R=\tset{D\in \mc D_n}{ {\rm E}^n_{D}[L(z^{\br}_{\train}(\obs), y)|x=\obs] > \sup_{R(D', D_{\train})\leq r} {\rm E}^n_{D'}[L(z^{\br}_{\train}(\obs), y)|x=\obs] }$ the disappointment set of a nominal formulation based on the cost estimator $(z, D)\mapsto {\rm E}^n_{D}[L(z, y)|x=\obs]=\E{D}{G(x, y)}$. Consequently, following inequality \eqref{eq:exponential-tightness} we have
  \[
    - \inf_{D' \in \mc N} \, B(D', D_{\train}) \leq \liminf_{n\to\infty} \frac{1}{n} \log D_{\train}^\infty\left[ D_{\bootstrap[n]} \in \mc N \right]\leq \liminf_{n\to\infty} \frac{1}{n} \log D_{\train}^\infty\left[ D_{\bootstrap[n]} \in \mc R \right]
  \]
  The stated result hence follows should we have $\inf_{D' \in \mc N} \, B(D', D_{\train})<r$. As $\mc N$ is an open set there exists $\lambda\in (0, 1)$ so that $D(\lambda)=\lambda D+(1-\lambda)D_{\train[n]}\in \mc N$. From convexity of $B$ we have $B(D(\lambda), D_{\train}) \leq \lambda B(D, D_{\train[n]}) + (1-\lambda) B(D_{\train}, D_{\train}) < r$ for any $\lambda\in (0,1)$ using that $B(D, D_{\train})=r$ and $B(D_{\train}, D_{\train})=0$. Hence, indeed we have $\inf_{D' \in \mc N} \, B(D', D_{\train})\leq B(D(\lambda), D_{\train}) < r$.
\end{proof}

\subsection{Proof of Corollary \ref{cor:nw_formulation}}
\label{ssec:corollary_nw_formulation}

\begin{proof}

  Remark that from the definition of the Nadaraya-Watson cost estimate $$\Es{D_{\train[n]}}{L(z, y)|x=\obs} \defn \tfrac{\E{D_{\train[n]}}{ L(z, y)\cdot w_n(x,\obs)}}{\E{D_{\train[n]}}{w_n(x,\obs)}}$$ given in \eqref{def:nw} it follows that we have
  \[
    \begin{array}{rl}
      \Es{D_{\train[n]}}{L(z, y)|x=\obs} \defn \max_{s>0, P} & \sum_{(\bar x, \bar y)\in \Omega_n}\, w_n(\bar x, \obs)\cdot L(z, \bar y)\cdot  P[\bar x, \bar y] \\[0.5em]
      \st  & P[\bar x, \bar y] = D[\bar x, \bar y] \cdot s \quad \forall (\bar x, \bar y) \in \Omega_n,\\[0.5em]
                                                              & \sum_{(\bar x, \bar y)\in\Omega_n} P[\bar x, \bar y] = s,~ \sum_{(\bar x, \bar y)\in \Omega_n} w_n(\bar x, \obs) \cdot P[\bar x, \bar y] = 1.
    \end{array}
  \]
  Indeed, the only feasible $s$ is such that $s = \tfrac{1}{\sum_{(\bar x, \bar y)\in \Omega_n} w_n(\bar x, \obs) \cdot P[\bar x, \bar y]}$. Hence, the only feasible $P$ is $P = \tfrac{D}{\sum_{(\bar x, \bar y)\in \Omega_n} w_n(\bar x, \obs) \cdot P[\bar x, \bar y]}$ and the equivalence follows. The chain of equalities
  \[
    \begin{array}{r@{}l}
  \set{(s, P)}{\exists D ~\st~ s\cdot D= P,\,R(D, D_{\train[n]})\leq r} & = \set{(s, P)}{R(\tfrac{P}{s}, D_{\train[n]}) \leq r} \\[0.5em]
                                                                      &  =  \set{(s, P)}{s\cdot R(\tfrac{P}{s}, D_{\train[n]}) \leq s\cdot r}    \end{array}
  \]
  imply that the robust budget function $c_{n}(z, D_{\train[n]}, \obs)$ corresponds exactly to the optimization formulation claimed in the corollary.
\end{proof}

\subsection{Proof of Theorem \ref{thm:consitencty-robust-nw}}
\label{ssec:consitencty-robust-nw}

\begin{proof}
  We first show the uniform convergence of the robust budget function to its nominal counterpart when the loss function $L(\bar z, \bar y) < \bar L < \infty$ is bounded. Let us first consider a given training data set $\train[n]$. Note that because of its definition as robust counterpart of our estimator, we have that our robust budget can be bounded as
  \[
    \Es{D_{\train[n]}}{L(\bar z, y)|x=\obs}\leq c_{n}(\bar z, D_{\train[n]}, \obs) \leq \Es{D_{\wc[n]}}{L(\bar z, y)|x=\obs}+\alpha
  \]
  for some worst-case distributions $D_{\wc[n]}$ at distance at most $B(D_{\wc[n]}, D_{\train[n]})\leq r(n)$ from the training distribution for any arbitrary $\alpha>0$. In terms of the total variation distance, we have that $\norm{D_{\wc[n]} - D_{\train[n]}}_1 \leq \sqrt{\tfrac{B(D_{\wc[n]} , D_{\train[n]})}{2}} \leq \sqrt{\tfrac{r(n)}{2}}$ following Pinkser's inequality.

 The Nadaraya-Watson estimate based on the worst-case distribution is the fraction
  \[
    \Es{D_{\wc[n]}}{L(\bar z, y)|x=\obs} \defn \frac{\textstyle\sum_{(\bar x,\bar y)\in \Omega_n} L(\bar z, \bar y) \cdot w_n(\bar x,\obs)\cdot D_{\wc[n]}[\bar x, \bar y]/h(n)^d}{\sum_{(\bar x,\bar y)\in \Omega_n} w_n(\bar x,\obs)\cdot D_{\wc[n]}[\bar x, \bar y]/h(n)^d}
  \]
  where we denote here $d=\dim(x)$ for conciseness.
  We have that the denominator of the Nadaraya-Watson estimator is lower bounded by
  \begin{align*}
    & \textstyle\sum_{(\bar x, \bar y)\in\Omega_n} \tfrac{w_n(\bar x, \obs)}{h(n)^d} D_{\wc[n]}[\bar x, \bar y]\\[0.5em]
    = & \textstyle \sum_{(\bar x, \bar y)\in\Omega_n} \tfrac{w_n(\bar x, \obs)}{h(n)^d} D_{\train[n]}[\bar x, \bar y]
         +\sum_{(\bar x, \bar y)\in\Omega_n} \tfrac{w_n(\bar x, \obs)}{h(n)^d}   \left(D_{\wc[n]}[\bar x, \bar y]-D_{\train[n]}[\bar x, \bar y]\right)\\[0.5em]
\geq  &\textstyle \sum_{(\bar x, \bar y)\in\Omega_n} \tfrac{w_n(\bar x, \obs)}{h(n)^d} D_{\train[n]}[\bar x, \bar y] - \left(\max_{(\bar x, \bar y)\in \Omega_n}w_n(\bar x, \obs)\right)\tfrac{\sqrt{\tfrac{r(n)}{2}}}{h(n)^d}\\[0.5em]
    \geq & \textstyle \sum_{(\bar x, \bar y)\in\Omega_n} \tfrac{w_n(\bar x, \obs)}{h(n)^d} D_{\train[n]}[\bar x, \bar y] - \tfrac{\sqrt{\tfrac{r(n)}{2}}}{h(n)^d}
  \end{align*}
  Here the first inequality follows from the Cauchy-Schwartz inequality $\abs{a\tpose b} \leq \norm{a}_\infty\cdot \norm{b}_1 $. Notice that here the weights $0\leq w_n(\bar x, \obs)\leq w_n(\obs, \obs)\leq 1$ are all non-negative and bounded from above by one for all smoother functions in Figure~\ref{fig:kernels}. Lemma 6 in \citet{walk2010strong} establishes that the limit $$\liminf_{n\to\infty} \textstyle \sum_{(\bar x, \bar y)\in\Omega_n} \tfrac{w_n(\bar x, \obs)}{h(n)^d} D_{\data[n]}[\bar x, \bar y] = 2d(x_0)>0$$ is positive with probability one. Taken together with the premise $\lim_{n\to\infty} \tfrac{\sqrt{r(n)}}{h(n)^d}=0$ this establishes the existence of a large enough sample size $n_0$ such that for all $n\geq n_0$ the denominator of the Nadaraya-Watson estimator satisfies $\sum_{(\bar x, \bar y)\in\Omega_n} \tfrac{w_n(\bar x, \obs)}{h(n)^d}  D_{\wc[n]}[\bar x, \bar y] \!\geq\! d(\obs)$ and $\sum_{(\bar x, \bar y)\in\Omega_n}\!\! \tfrac{w_n(\bar x, \obs)}{h(n)^d} D_{\train[n]}[\bar x, \bar y]\geq 2\tfrac{\sqrt{\tfrac{r(n)}{2}}}{h(n)^d}$. Similarly, the nominator of the Nadaraya-Watson estimator satisfies
  \[
    \begin{array}{rl}
      & \sum_{(\bar x, \bar y)\in\Omega_n} L(\bar z, \bar y)\cdot w_n(\bar x, \obs) \cdot D_{\wc[n]}[\bar x, \bar y]/h(n)^d\\[0.5em]
      \leq  &\sum_{(\bar x, \bar y)\in\Omega_n} L(\bar z, \bar y)\cdot w_n(\bar x, \obs)\cdot D_{\train[n]}[\bar x, \bar y]/h(n)^d + \sqrt{\tfrac{r(n)}{2}}/h(n)^d.
    \end{array}
  \]
  In what follows we will use the inequality $\tfrac{a}{(b-x)} \leq \tfrac{a}{b} + \tfrac{2a}{b} \cdot x$ for all $x\leq \tfrac{b}{2}$ when $a, b>0$. This inequality follows trivially from the definition of convexity of the function $\tfrac{a}{(b-x)}$ for all $x\leq b$ when $a, b>0$.  Using the previous inequalities we can establish when $n\geq n_0$ the following claims 
  \begin{align*}
    & \Es{D_{\wc[n]}}{L(\bar z, y)|x=\obs} \\
    & \leq \frac{\textstyle\sum_{(\bar x,\bar y)\in \Omega_n} L(\bar z, \bar y) \cdot w_n(\bar x,\obs)\cdot D_{\train[n]}[\bar x, \bar y]/h(n)^d}{\sum_{(\bar x,\bar y)\in \Omega_n} w_n(\bar x,\obs)\cdot D_{\wc[n]}[\bar x, \bar y]/h(n)^d} +  \frac{\sqrt{\tfrac{r(n)}{2}}}{h(n)^d d(x_0)} \\
                                     & \leq \frac{\textstyle\sum_{(\bar x,\bar y)\in \Omega_n} L(\bar z, \bar y) \cdot w_n(\bar x,\obs)\cdot D_{\train[n]}[\bar x, \bar y]/h(n)^d}{\sum_{(\bar x,\bar y)\in \Omega_n} w_n(\bar x,\obs)\cdot D_{\train[n]}[\bar x, \bar y]/h(n)^d-\tfrac{\sqrt{\tfrac{r(n)}{2}}}{h(n)^d}} +  \frac{\sqrt{\tfrac{r(n)}{2}}}{h(n)^d d(x_0)}\\
                                     & \leq \Es{D_{\train[n]}}{L(\bar z, y)|x=\obs} + 2\Es{D_{\train[n]}}{L(\bar z, y)|x=\obs}\tfrac{\sqrt{\tfrac{r(n)}{2}}}{h(n)^d} +  \frac{\sqrt{\tfrac{r(n)}{2}}}{h(n)^d d(x_0)}\\
    &\leq \Es{D_{\train[n]}}{L(\bar z, y)|x=\obs} + (2\bar L +  \tfrac{1}{d(x_0)}) \tfrac{\sqrt{\tfrac{r(n)}{2}}}{h(n)^d}.
  \end{align*}
  The nominal Nadaraya-Watson estimator is uniformly consistent, i.e., $$|\Es{D_{\data[n]}}{L(\bar z, y)|x=\obs} - \E{D^\star}{L(\bar z, y)|x=\obs}|\leq \epsilon(n)$$ for $\lim_{n\to\infty} \epsilon(n) = 0$ as discussed before. It hence trivially follows that the robust Nadaraya-Watson estimator is uniformly consistent as well. Indeed, from the previous inequality it follows that
  \[
    |c_{n}(\bar z, D_{\data[n]}, \obs) - \E{D^\star}{L(\bar z, y)|x=\obs}|\leq \epsilon(n) + (2\bar L +  \tfrac{1}{d(x_0)}) \tfrac{\sqrt{\tfrac{r(n)}{2}}}{h(n)^d} + \alpha
  \]
  with probability one for all $\bar z$. This inequality holds for any arbitrary $\alpha>0$.
  Uniform consistency then directly implies  here an asymptotically diminishing optimality gap  $$\E{D^\star}{L(z_{\data[n]}^r, y)|x=\obs} - \min_{z}\, \E{D^\star}{L(\bar z, y)|x=\obs} \leq 2\epsilon(n) + (4\bar L +  \tfrac{2}{d(x_0)}) \tfrac{\sqrt{\tfrac{r(n)}{2}}}{h(n)^d}.$$
  Using that $\lim_{n\to\infty} \tfrac{\sqrt{r(n)}}{h(n)^d}$ yields the wanted result immediately.
\end{proof}

\subsection{Proof of Theorem \ref{thm:consitencty-robust-nn}}
\label{ssec:consitencty-robust-nn}

\begin{proof}
  We show the uniform convergence of the robust budget function to the unknown cost, that is,  and any bounded function $L(\bar z, \bar y) < \bar L < \infty$ for all $\bar z$ and $\bar y$. Let us first consider a given training data set $\train[n]$ without ties. That is, we have that $\abs{N^j_n(\obs)}=j$ for all $j\in [n]$. Note that because of its definition as robust counterpart of the estimator $\Es{D_{\train[n]}}{L(\bar z, y)|x=\obs}$, we have that the robust cost can be bounded as $$\Es{D_{\train[n]}}{L(\bar z, y)|x=\obs}\leq c_{n}(z^r_{\train[n]}, D_{\train[n]}, \obs) \leq {\rm{E}}_{D_{\wc[n]}}^{j^\star}[L(\bar z, y)|x=\obs]+\alpha$$ for some worst-case distributions $D_{\wc[n]}\in \mc D^{j\star}_n$ at distance at most $B(D_{\wc[n]}, D_{\train[n]})\leq r(n)$ from the training distribution for any arbitrary $\alpha>0$. In terms of the total variation distance, we have that $\norm{D_{\wc[n]} - D_{\train[n]}}_1 \leq \sqrt{\tfrac{B(D_{\wc[n]} , D_{\train[n]})}{2}} \leq \sqrt{\tfrac{r(n)}{2}}$ following Pinkser's inequality. The nearest-neighbors estimate based on the worst-case distribution is defined as the fraction
  \[
    \Es{D_{\wc[n]}}{L(\bar z, y)|x=\obs} \defn \frac{\textstyle\sum_{(\bar x,\bar y)\in N^{j\star}_n(\obs)} L(\bar z, \bar y) \cdot D_{\wc[n]}[\bar x, \bar y]}{\sum_{(\bar x,\bar y)\in N^{j\star}_n(\obs)} D_{\wc[n]}[\bar x, \bar y]}.
  \]
  The neighborhood parameter $j^\star$ satisfies by definition
  \(
    \textstyle\sum_{(\bar x, \bar y) \in N^{j\star}_n(\obs)} D_{\wc[n]}[\bar x, \bar y] \geq k(n)/n.
  \)
  The previous inequality bounds the denominator from below by $\tfrac{k(n)}{n}$. Using the Cauchy-Schwartz inequality $\abs{a\tpose b} \leq \norm{a}_\infty\cdot \norm{b}_1 $ as well as the Pinkser inequality, the nominator of the Nadaraya-Watson estimator satisfies
  \[
    \begin{array}{rl}
      & \sum_{(\bar x, \bar y)\in N^{j\star}_n(\obs)} L(\bar z, \bar y) \cdot D_{\wc[n]}[\bar x, \bar y]\\[0.5em]
      \leq & \sum_{(\bar x, \bar y)\in N^{j\star}_n(\obs)} L(\bar z, \bar y) \cdot D_{\train[n]}[\bar x, \bar y]+ \bar L  \sqrt{\tfrac{r(n)}{2}}.
    \end{array}
  \]
  We also have that from the definition of the total variation distance that $\Vert D_{\train[n]}-D_{\wc[n]}\Vert_1\geq D_{\train[n]}[N^{j\star}_n(\obs)] - D_{\wc[n]}[N^{j\star}_n(\obs)] \geq \tfrac{(j^\star-k(n))}{n}$. We have also $\Vert D_{\wc[n]}-D_{\train[n]}\Vert_1 \geq D_{\wc[n]}[N^{j\star-1}_n(\obs)] - D_{\train[n]}[N^{j\star-1}_n(\obs)] \geq \tfrac{(k(n)-j^\star)}{n}$. Last two inequalities imply that we can use the bound $\sqrt{r(n)/2}\geq\Vert D_{\train[n]}-D_{\wc[n]}\Vert_1 \geq \abs{k(n)-j^\star}/n$. By applying first the Cauchy-Schwartz inequality again and then the previously obtained bounds we can obtain
  \[
    \begin{array}{rl}
      & \sum_{(\bar x, \bar y)\in N^{j\star}_n(\obs)} L(\bar z, \bar y) D_{\train[n]}[\bar x, \bar y]\\[0.5em]
      \leq &  \sum_{(\bar x, \bar y)\in N^{k(n)}_n(\obs)} L(\bar z, \bar y) D_{\train[n]}[\bar x, \bar y] + \bar L | D_{\train[n]}\{N^{j\star}_n(\obs)]-D_{\train[n]}[N^{k(n)}_n(\obs)]|\\[0.5em]
      \leq & \sum_{(\bar x, \bar y)\in N^{k(n)}_n(\obs)} L(\bar z, \bar y) D_{\train[n]}[\bar x, \bar y] + \bar L \abs{j^\star-k(n)}/n\\[0.5em]
      \leq & \sum_{(\bar x, \bar y)\in N^{k(n)}_n(\obs)} L(\bar z, \bar y) D_{\train[n]}[\bar x, \bar y] + \bar L  \sqrt{\tfrac{r(n)}{2}}.
    \end{array}
  \]
  Hence,
  \begin{equation}
    \label{eq:ineq-sandwich}
    \Es{D_{\train[n]}}{L(\bar z, y)|x=\obs}\leq c_n(\bar z, D_{\train[n]}, \obs)
    \leq \Es{D_{\train[n]}}{L(\bar z, y)|x=\obs} + \frac{2 \bar L  n\sqrt{\tfrac{r(n)}{2}}}{k(n)}+\alpha
  \end{equation}
  for any arbitrary training data set without ties. Ties among data points when using the random tie breaking method are a probability zero event which we may ignore.
  We already know that the nearest neighbors estimator is uniformly consistent. That is, we have that $|\Es{D_{\data[n]}}{L(\bar z, y)|x=\obs} - \E{D^\star}{L(\bar z, y)|x=\obs}| \leq \epsilon(n)$ with probability one and $\lim_{n\to\infty} \epsilon(n)=0$. When the robustness radius does shrinks at an appropriate rate, i.e., its size compared to the bandwidth parameter is negligible ($\lim_{n\to\infty} \tfrac{n\sqrt{r(n)}}{k(n)}=0$), then uniform consistency of budget estimator $c_n$ follows by taking the limit for $n$ tends to infinity for the chain of inequalities in \eqref{eq:ineq-sandwich} applied to $D_{\data[n]}$ and observing that $\alpha>0$ is arbitrarily small. Uniform consistency of the nearest-neighbors formulation follows by the exact same argument as given in proof of Theorem \ref{thm:consitencty-robust-nw} in case of the Nadaraya-Watson formulation.
\end{proof}

\subsection{Proof of Corollary~\ref{thm:bootstrap-nw}}
\label{ssec:proof-corollary}

\begin{proof}
  Let us fix a training data set with empirical training distribution $D_{\train[n]}$ and a given decision $z$. Let $\bar c_n\defn\Es{D_{\train[n]}}{L(\bar z, y)|x=\obs}$ be the budgeted cost based on the training data with $k(n)=n$. In order to prove the theorem, it suffices to characterize the probability of the event that the empirical distribution $D_{\bootstrap[n]}$ of random bootstrap data resampled from the training data realizes in the set
  \begin{equation*}
    \mc C \defn  \set{D \in \mc D_n}{
      \begin{array}{l} \exists s >0, ~\textstyle s \cdot \sum_{(\bar x, \bar y)\in \Omega_n} w_n(\bar x, \obs)\cdot L(\bar z, y) \cdot D(\bar x, \bar y) >  \bar c_n,\\[0.5em] s\cdot \sum_{(\bar x, \bar y)\in \Omega_n} w_n(\bar x,\obs) \cdot D(\bar x, \bar y)=1
      \end{array}
    }
  \end{equation*}
  as follows from Corollary \ref{cor:nw_formulation}.
  After eliminating the auxiliary variable $s$ we arrive at the description $\mc C =  \{D\in \mc D_n : \textstyle \sum_{(\bar x, \bar y)\in\Omega_n} w_n(\bar x, \obs)\cdot L(\bar z, y) \cdot D(\bar x, \bar y) >  \bar c_n \cdot \sum_{(\bar x, \bar y)\in\Omega_n} w_n(\bar x,\obs) \cdot D(\bar x, \bar y)\}$. The set $\mc C$ is a convex polyhedron. The robust budget cost $\bar c_n$ is constructed to ensure that $\inf_{D\in \mc C}\, R(D, D_{\train[n]}) > r$. Indeed, we have the rather direct implication $\bar{D} \in \mc C \!\implies\! \Es{\bar D}{L(\bar z, y)|x=\obs} > \bar c_n = \sup \,\{ \Es{ D}{L(\bar z, y)|x=\obs} \,|\, R(D, D_{\train[n]})\leq r \}$ which in turn itself implies $R(\bar D, D_{\train[n]}) > r$.  Hence, the result follows from the bootstrap inequality \eqref{eq:ldb} applied to the probability $D_{\train[n]}^\infty(D_{\bootstrap[n]} \in \mc C)$ as in this particular case the employed distribution distance function ($R=B$) coincides with the bootstrap distance function.
\end{proof}

\subsection{Proof Lemma \ref{lemma:nn-dual-representation}}
\label{ssec:proof-dual}

\begin{proof}
  We will employ standard Lagrangian duality on the convex optimization characterization \eqref{eq:nna:primal} of the partial nearest neighbors cost function associated $c^j_n(\bar z, D, \obs)$. The Lagrangian function associated with the primal optimization problem in \eqref{eq:nna:primal} is denoted here at the function 
  \begin{align*}
    & \mc L(P, s; \alpha, \beta, \eta, \nu) \defn \\
    & \textstyle \sum_{(\bar x, \bar y)\in N^j_{n}(\obs)} w_n(\bar x, \obs) \cdot  L(z, \bar y) \cdot P[x, y] \textstyle  + \left(1-\sum_{(\bar x, \bar y)\in N^{j-1}_{n}(\obs)} w_n(\bar x, \obs) \cdot P[\bar x, \bar y]\right)\alpha   \\
    & \textstyle + \left( \sum_{(\bar x, \bar y)\in N^j_{n}(\obs)} P[\bar x, \bar y] - \frac{k}{n} \cdot s \right) \eta_1 + \left( \frac{k-1}{n}  \cdot s - \sum_{(\bar x, \bar y)\in N^{j-1}_{n}(\obs)} P[\bar x, \bar y]  \right) \eta_2 \\
    & \textstyle + \left(\sum_{(\bar x, \bar y)\in \Omega_n} P[\bar x, \bar y] -s\right) \beta + \left( r \cdot s - \sum_{(\bar x, \bar y)\in \Omega_n} P[\bar x, \bar y] \log\left(\frac{P[\bar x, \bar y]}{ s \cdot D[\bar x, \bar y]}\right) \right) \nu 
  \end{align*}
  where $P$ and $s$ are the primal variables of the primal optimization problem \eqref{eq:nna:primal} and $\alpha$, $\beta$, $\eta$ and $\nu$ the dual variables associated with each of its constraints. Collecting the relevant terms in the Lagrangian function results in $\mc L(P, s; \alpha, \beta, \nu) =$
  \begin{align*}
    \textstyle\alpha +& \textstyle s (r \nu - \beta - \frac{k}{n}(\eta_1-\eta_2) - \frac{\eta_2}{n}) \\
                      & \textstyle+ \sum_{(\bar x, \bar y)\in N^{j-1}_{n}(\obs)} \left[P[\bar x, \bar y] \left((L(z, y)-\alpha)\cdot w_n(\bar x, \obs) +\beta + \eta_1 -\eta_2\right) - \nu P[\bar x, \bar y] \log\left(\frac{P[\bar x, \bar y]}{ s \cdot D[\bar x, \bar y]}\right)\right]\\
                      & \textstyle+ \sum_{(\bar x, \bar y)\in N^j_{n}(\obs)\setminus N^{j-1}_{n}(\obs)} \left[P[\bar x, \bar y] \left((L(z, y)-\alpha)\cdot w_n(\bar x,\obs) +\beta + \eta_1\right) - \nu P[\bar x, \bar y] \log\left(\frac{P[\bar x, \bar y]}{ s \cdot D[\bar x, \bar y]}\right)\right]\\
                      & \textstyle+ \sum_{(\bar x, \bar y)\in \Omega_n\setminus \ms N^j_{n}(\obs)} \left[P[\bar x, \bar y] \beta - \nu P[\bar x, \bar y] \log\left(\frac{P[\bar x, \bar y]}{ s \cdot D[\bar x, \bar y]}\right)\right]
  \end{align*}
  The dual function of the primal optimization problem \eqref{eq:nna:primal} is identified with the concave function $g(\alpha, \beta, \eta, \nu) \defn \inf_{P\geq 0,\,s> 0} \mc L(P, s; \alpha, \beta, \nu)$. Using the same manipulations as presented in the proof of Lemma \ref{lemma:nw-dual-representation} we can express the dual function as $g(\alpha, \beta, \eta, \nu) =$
  \begin{align*}
    =  \textstyle\sup_{s> 0} \,  \alpha + s (r \nu - & \textstyle \beta - \frac{k}{n}(\eta_1-\eta_2) - \frac{\eta_2}{n}) +s \nu \sum_{(\bar x, \bar y)\in \Omega_n\setminus N^j_{n}(\obs)} D[\bar x, \bar y] \exp\left(\frac{ \beta }{\nu}-1\right) \\
                                & \textstyle +s \nu \sum_{(\bar x, \bar y)\in N^{j-1}_{n}(\obs)} D[\bar x, \bar y] \exp\left(\frac{(L(z, \bar y)-\alpha) \cdot w_n(\bar x, \obs) + \beta + \eta_1-\eta_2}{\nu}-1\right) \\
                                & \textstyle +s \nu \sum_{(\bar x, \bar y)\in N^j_{n}(\obs)\setminus N^{j-1}_{n}(\obs)} D[\bar x, \bar y] \exp\left(\frac{(L(z, \bar y)-\alpha) \cdot w_n(\bar x,\obs) + \beta + \eta_1}{\nu}-1\right).
  \end{align*}
  Our dual function can be expressed alternatively as 
  \begin{align*}
    g(\alpha, \beta, \eta, \nu) =  \textstyle \Big\{\alpha : r \cdot \nu & \textstyle - \frac{k}{n}(\eta_1-\eta_2) - \frac{\eta_2}{n} + \nu \sum_{(\bar x, \bar y)\in\Omega_n\setminus N^j_{n}(\obs)} D[\bar x, \bar y] \exp\left(\frac{ \beta }{\nu}-1\right)   \\
                                                               & \textstyle + \nu \sum_{(\bar x, \bar y)\in N^{j-1}_{n}(\obs)} D[\bar x, \bar y] \exp\left(\frac{(L(z, \bar y)-\alpha) \cdot w_n(\bar x, \obs) + \beta + \eta_1-\eta_2}{\nu}-1\right)   \\
                                                               & \textstyle + \nu \sum_{(\bar x, \bar y)\in N^j_{n}(\obs)\setminus N^{j-1}_{n}(\obs)} D[\bar x, \bar y] \exp\left(\frac{(L(z, \bar y)-\alpha) \cdot w_n(\bar x, \obs) + \beta + \eta_1}{\nu}-1\right) \leq \beta \Big\} .
  \end{align*}
  The dual optimization problem of the primal problem \eqref{eq:nna:primal} is now found as $\inf_{\alpha, \beta, \nu\geq 0}\, g(\alpha, \beta, \nu)$. As the primal optimization problem in \eqref{eq:nna:primal} is convex, strong duality holds under Slater's condition which is satisfied whenever $r>r^j_n$.  Using first-order optimality conditions, the optimal $\beta^\star$ must satisfy the relationship $\beta^\star =  -\nu + \nu \log (\sum_{(\bar x, \bar y) \in N^{j-1}_{n}(\obs} D[\bar x, \bar y] \exp([(L(z, \bar y)-\alpha) \cdot w_n(\bar x, \obs)  +\eta_1-\eta_2]/\nu)  + \sum_{(\bar x, \bar y) \in N^j_{n}(\obs)\setminus N^{j-1}_{n}(\obs)} D[\bar x, \bar y] \exp([(L(z, \bar y)-\alpha) \cdot w_n(\bar x, \obs) +\eta_1] /\nu) + \sum_{(\bar x, \bar y) \in N^j_{n}(\obs)\setminus N^{j-1}_{n}(\obs)} D[\bar x, \bar y])$.
  Substituting the optimal value of $\beta^\star$ in the back in the dual optimization problem gives
  \begin{align*}
    \inf_{\alpha, \beta, \nu\geq 0}\, g(\alpha, \beta, \eta, \nu)  = & \textstyle\inf_{\alpha, \nu\geq 0}\, g(\alpha, \beta^\star, \eta, \nu) \\
            = & \textstyle \inf \Big\{\alpha\in \Re : \exists \nu\in\Re_+, \exists \eta \in \Re_+^2, ~r\cdot \nu - \frac{k}{n}(\eta_1-\eta_2) - \frac{\eta_2}{n} \cdot \nu \\
            & \qquad + \nu \log (\textstyle\sum_{(\bar x, \bar y)\in N^{j-1}_n(\obs)} \exp({[(L(z, \bar y)-\alpha)\cdot w_n(\bar x, \obs) + \eta_1-\eta_2]}/{\nu})\cdot D[\bar x, \bar y] \\[0.5em]
            & \qquad \textstyle + \sum_{(\bar x, \bar y)\in N^j_n(\obs)\setminus N^{j-1}_n(\obs)} \exp({[(L(z, \bar y)-\alpha)\cdot w_n(\bar x,\obs) + \eta_1]}/{\nu})\cdot D[\bar x, \bar y] \\[0.5em]
            & \qquad \textstyle + \sum_{\Omega_n \setminus N^j_n(\obs)} D[\bar x, \bar y] ) \leq 0 \Big\}.
  \end{align*}
\end{proof}

\subsection{Proof Lemma \ref{lemma:nw-dual-representation}}
\label{ssec:proof-dual-nw}

\begin{proof}
  We will employ standard Lagrangian duality on the convex optimization characterization of the Nadaraya-Watson cost function given in Corollary \ref{cor:nw_formulation}. The Lagrangian function associated with the primal optimization problem  is denoted here at the function 
  \begin{align*}
   \mc L(P, s; \alpha, \beta, \nu) \defn  \textstyle \sum_{(\bar x, \bar y)\in \Omega_n} w_n(\bar x, \obs) & \cdot  L(\bar z, \bar y) \cdot P[\bar x, \bar y] \textstyle  + \left(1-\sum_{(\bar x, \bar y)\in \Omega_n} w_n(\bar x, \obs) \cdot P[\bar x, \bar y]\right)\alpha +  \\
    & \textstyle\left(\sum_{(\bar x, \bar y)\in \Omega_n} P[\bar x, \bar y] -s\right) \beta + \left( r \cdot s - \sum_{(\bar x, \bar y)\in \Omega_n} P[\bar x, \bar y] \log\left(\frac{P[\bar x, \bar y]}{ s \cdot D[\bar x, \bar y]}\right) \right) \nu 
  \end{align*}
  where $P$ and $s$ are the primal variables of the primal optimization problem given in Corollary \ref{cor:nw_formulation}  and $\alpha$, $\beta$ and $\nu$ the dual variables associated with each of its constraints. Collecting the relevant terms in the Lagrangian function results in
  \begin{equation*}
   \mc L(P, s; \alpha, \beta, \nu) = \textstyle\alpha + s (r \nu - \beta) + \sum_{(\bar x, \bar y)\in \Omega_n} \left[P[\bar x, \bar y] \left((L(\bar z, \bar y)-\alpha)\cdot w_n(\bar x, \obs) +\beta\right) - \nu P[\bar x, \bar y] \log\left(\frac{P[\bar x, \bar y]}{ s \cdot D[\bar x, \bar y]}\right)\right]
  \end{equation*}
   The dual function of the primal optimization problem is identified with $g(\alpha, \beta, \nu) \defn \inf_{P\geq 0,\,s> 0} \mc L(P, s; \alpha, \beta, \nu)$. Our dual function can be expressed alternatively as $g(\alpha, \beta, \nu) =$
  \begin{align*}
    & \textstyle \sup_{s> 0} \, \alpha + s\left(r \nu - \beta \right) + {\displaystyle\sup_{P\geq 0}\, \sum_{(\bar x, \bar y)\in \Omega_n}} \left[ P[\bar x, \bar y] \left((L(\bar z, \bar y)-\alpha) \cdot w_n(\bar x,\obs) +\beta\right) - \nu P[\bar x, \bar y] \log\left(\frac{P[\bar x, \bar y]}{ s\cdot D[\bar x, \bar y]}\right)\right] \\
    = & \textstyle \sup_{s> 0} \, \alpha + s\left(r \nu - \beta \right) + {\displaystyle\sum_{(\bar x, \bar y)\in \Omega_n} \sup_{P[\bar x, \bar y]\geq 0}} \left[P[\bar x, \bar y] \left((L(\bar z, \bar y)-\alpha) \cdot w_n(\bar x, \obs) +\beta\right) - \nu P[\bar x, \bar y] \log\left(\frac{P[\bar x, \bar y]}{ s \cdot D[\bar x, \bar y]}\right)\right] \\
    = & \textstyle \sup_{s> 0} \, \alpha + s\left(r \nu - \beta \right) + s \sum_{(\bar x, \bar y)\in \Omega_n} D[\bar x, \bar y][ \sup_{\lambda\geq 0} \, \lambda \left((L(\bar z, \bar y)-\alpha)\cdot w_n(\bar x, \obs) + \beta\right) - \nu \lambda \log\left(\lambda\right)]. \\
    \intertext{The inner maximization problems over $\lambda$ can be dealt with using the Fenchel conjugate of the $\lambda\mapsto\lambda \cdot \log \lambda$ function as}
    = & \textstyle\sup_{s> 0} \, \alpha + s\left(r \nu - \beta \right) +s \nu \sum_{(\bar x, \bar y)\in \Omega_n} D[\bar x, \bar y] \exp\left(\frac{(L(\bar z, \bar y)-\alpha) \cdot w_n(\bar x, \obs) + \beta}{\nu}-1\right)\\
    = & \textstyle\set{\alpha}{r\nu  + \nu \sum_{(\bar x, \bar y)\in \Omega_n} D[\bar x, \bar y] \exp\left(\frac{(L(\bar z, \bar y)-\alpha) \cdot w_n(\bar x, \obs) + \beta}{\nu}-1\right) \leq \beta}.
  \end{align*}
  The dual optimization problem is now found as $\inf_{\alpha, \beta, \nu\geq 0}\, g(\alpha, \beta, \nu)$. As our primal optimization is convex, strong duality holds under Slater's condition which is satisfied whenever $r>0$.  Using first-order optimality conditions, the optimal $\beta^\star$ must satisfy 
  \(
  \beta^\star =  -\nu + \nu \log (\sum_{(\bar x, \bar y)\in \Omega_n} D[\bar x, \bar y] \exp((L(\bar z, \bar y)-\alpha) \cdot w_n(\bar x, \obs)/\nu)).
  \)
  Substituting the optimal value of $\beta^\star$ in the back in the dual optimization problem gives
  \begin{align*}
    \inf_{\alpha, \beta, \nu\geq 0}\, g(\alpha, \beta, \nu)  & = \textstyle\inf_{\alpha, \nu\geq 0}\, g(\alpha, \beta^\star, \nu) \\
    & = \textstyle \inf \set{\alpha\in \Re}{\exists \nu\in\Re_+, \,r\nu + \nu \log\left(\sum_{(\bar x, \bar y)\in \Omega_n} D[\bar x, \bar y] \exp\left(\frac{(L(\bar z, \bar y)-\alpha)\cdot w_n(\bar x,\obs)}{\nu}\right)\right)\leq 0}.
  \end{align*}
\end{proof}

\end{document}

%% file: preamble.tex
\usepackage[margin=0.8in]{geometry}

\usepackage{amsmath, amsthm, amssymb, bbold, mathtools}
\usepackage{mathrsfs}
\usepackage{footnote}
\usepackage{graphicx}
\usepackage{subcaption}
\usepackage{acronym}
\usepackage{multirow}
\usepackage{booktabs}
\usepackage{tabularx}

\usepackage{enumerate}
\usepackage{hyperref}
\usepackage{enumitem}
\usepackage{float,lscape}
\usepackage[mathscr]{eucal}

\newtheorem{proposition}{Proposition}
\newtheorem{lemma}{Lemma}
\newtheorem{definition}{Definition}
\newtheorem{assumption}{Assumption}
\newtheorem{theorem}{Theorem}

\newtheorem{remark}{Remark}
\newtheorem{corollary}{Corollary}

\acrodef{var}[VAR]{value-at-risk}

\acrodef{saa}[SAF]{sample average formulation}
\acrodef{sla}[SLF]{supervised learning formulation}
\acrodef{nwa}[NWF]{Nadaraya-Watson formulation}
\acrodef{nna}[NNF]{nearest neighbors formulation}
\acrodef{lska}[LSKA]{least squares kernel formulation}
\acrodef{cart}[CART]{classification and regression trees}

\setlength{\parindent}{0pt}
\setlength{\parskip}{5pt}
\makeatletter
\def\thm@space@setup{%
\thm@preskip=1em \thm@postskip=0pt
}
\makeatother

\usepackage{tikz}
\usetikzlibrary{decorations.pathmorphing} 
\usetikzlibrary{matrix} 
\usetikzlibrary{arrows} 
\usetikzlibrary{calc} 
\usetikzlibrary{positioning}
\usetikzlibrary{shapes}

\tikzstyle{block} = [draw,rectangle,minimum height=3em,minimum width=8em]
\tikzstyle{connector} = [->,thick]
\tikzstyle{line} = [thick]
\tikzstyle{branch} = [circle,inner sep=0pt,minimum size=1mm,fill=black,draw=black]
\tikzstyle{guide} = []


\usepackage{pgfplots}
\pgfplotsset{compat=newest}


\newcommand{\norm}[1]{\left\|#1\right\|}
\newcommand{\E}[2]{\mb{E}_{#1\!\!}\left[#2\right]}
\newcommand{\Es}[2]{{\mathrm{E}}^n_{#1\!\!}\left[#2\right]}
\newcommand{\Esp}[2]{{\mathrm{E}}^{n,j}_{#1}\left[#2\right]}

\newcommand{\B}[2]{{B}\!\left(#1, #2\right)}
\newcommand{\abs}[1]{\left|#1\right|}
\newcommand{\one}[0]{\mathbb{1}}

\newcommand{\set}[2]{\left\{ #1\ : \ #2 \right\}}
\newcommand{\tset}[2]{\{ #1\ : \ #2 \}}
\newcommand{\tpose}{^\top}
\newcommand{\defn}[0]{:=}

\newcommand{\mc}{\mathcal}
\newcommand{\mb}{\mathbb}
\newcommand{\mf}{\mathtt}
\newcommand{\ms}{\mf}
\newcommand{\mr}{\mathrm}

\renewcommand{\tfrac}[2]{{#1}/{#2}}

\newcommand{\validate}{\mr{vd}}

\newcommand{\bootstrap}{\mr{bs}}
\newcommand{\train}{\mr{tr}}

\newcommand{\iid}{\sim_{\rm{iid}}}
\newcommand{\data}{\mr{data}}
\newcommand{\wc}{\mr{wc}}

\newcommand{\obs}{{x_0}}

\def\dist{\mr{dist}}
\def\erf{\mr{erf}}

\def\Re{\mr{R}}

\def\st{\mr{s.t.}}

\def\bs{\mr{bs}}
\def\lin{\mr{lin}}
\def\br{\mr{r}}

\DeclareMathOperator{\dom}{dom}

\DeclareMathOperator{\interior}{int}

\DeclarePairedDelimiter\ceil{\lceil}{\rceil}

\usepackage[]{hyperref}
\hypersetup{
  colorlinks=true,
  citecolor=blue
}

\usepackage{tcolorbox}
\tcbuselibrary{skins}

\newtcolorbox{mybox}[2][]{
  colframe=black,
  colback=white,
  fonttitle=\bfseries,
  colbacktitle=white,
  coltitle=black,
  title=#2,
  #1}

\usepackage{tikz}
\usepackage{pgfplots}
\usetikzlibrary{shapes,arrows}
\tikzstyle{block} = [draw, rectangle, minimum height=3em, minimum width=6em]
\tikzstyle{input} = [coordinate]
\tikzstyle{output} = [coordinate]
